\documentclass[11pt,reqno,a4paper]{amsart}

\usepackage[utf8]{inputenc}
\usepackage{graphicx}
\usepackage{color}
\usepackage{transparent}

\usepackage{dsfont}
\usepackage{gensymb}
\usepackage[active]{srcltx}
\usepackage[colorlinks,pdfpagelabels,pdfstartview = FitH,bookmarksopen = true,bookmarksnumbered = true,linkcolor = blue,plainpages = false,hypertexnames = false,citecolor = red] {hyperref}
\usepackage{mathrsfs}
\usepackage{amssymb,amsmath, amsfonts, amsthm}
\usepackage{latexsym}
\usepackage[dvips]{epsfig}
\newtheorem{theorem}{Theorem}[section]
\newtheorem{proposition}[theorem]{Proposition}
\newtheorem{lemma}[theorem]{Lemma}
\newtheorem{corollary}[theorem]{Corollary}
\newtheorem{remark}[theorem]{Remark}



\usepackage{epsfig,color,xcolor}

\newcommand{\ble}{\begin{lemma}}
\newcommand{\ele}{\end{lemma}}
\newcommand{\be}{\begin{equation*}}
\newcommand{\ee}{\end{equation*}}

\newcommand{\bel}{\begin{equation}}
\newcommand{\eel}{\end{equation}}

\newcommand{\fr}{\frac }

\newcommand{\lap}{\Delta}
\newcommand{\N}{\mathbb{N}}
\newcommand{\na}{\nabla}

\newcommand{\R}{\mathbb{R}}

\renewcommand{\to}{\rightarrow}
\newcommand{\To}{\longrightarrow}

\newcommand{\xip}{x_{i,p}}
\newcommand{\xjp}{x_{j,p}}

\newcommand{\mip}{\mu_{i,p}}
\newcommand{\mjp}{\mu_{j,p}}

\newcommand{\upp}{u_p}

\def\sideremark#1{\ifvmode\leavevmode\fi\vadjust{\vbox to0pt{\vss% the remark
 \hbox to 0pt{\hskip\hsize\hskip1em%                          will appear only
 \vbox{\hsize2.1cm\tiny\raggedright\pretolerance10000%          on the side
  \noindent #1\hfill}\hss}\vbox to15pt{\vfil}\vss}}}%

\numberwithin{equation}{section}
\begin{document}

\title[]{Asymptotic profile of positive solutions of Lane-Emden problems in dimension two}

\author[]{Francesca De Marchis, Isabella Ianni, Filomena Pacella}

\address{Francesca De Marchis, University of Roma {\em Sapienza}, P.le Aldo Moro 5, 00185 Roma, Italy}
\address{Isabella Ianni, Second University of Napoli, V.le Lincoln 5, 81100 Caserta, Italy}
\address{Filomena Pacella, University of Roma {\em Sapienza}, P.le Aldo Moro 5, 00185 Roma, Italy}



\thanks{Research partially supported by: PRIN $201274$FYK7$\_005$ grant, INDAM - GNAMPA and Fondi Avvio alla Ricerca Sapienza 2015}

\begin{abstract} We  consider families $u_p$ of solutions to the  problem
\begin{equation}\label{problemAbstract}\left\{\begin{array}{lr}-\Delta u= u^p & \mbox{ in }\Omega\\
u>0 & \mbox{ in }\Omega\\
u=0 & \mbox{ on }\partial \Omega
\end{array}\right.\tag{$\mathcal E_p$}
\end{equation}
where $p>1$ and $\Omega$ is a smooth bounded domain of $\R^2$. Under the condition
\begin{equation}\label{EnergyBoundAbstract}p\int_{\Omega}|\nabla u_p|^2\,dx\to\beta\in\R\quad\mbox{as $p\to+\infty$}\tag{$\mathcal F$}
\end{equation} we give a complete description of the asymptotic behavior of  $u_p$ as $p\to+\infty$.
\end{abstract}

\maketitle

\section{Introduction}\label{section:intro}
We consider the Lane-Emden Dirichlet problem
\begin{equation}\label{problem}\left\{\begin{array}{lr}-\Delta u= |u|^{p-1}u\qquad  \mbox{ in }\Omega\\
u=0\qquad\qquad\qquad\mbox{ on }\partial \Omega
\end{array}\right.
\end{equation}
where $p>1$ and $\Omega\subset\R^2$ is a smooth bounded domain.

The aim of this paper is to provide a precise description of the asymptotic behavior, as $p\to+\infty$, of positive solutions of \eqref{problem} under a uniform bound of their energy, namely we consider any family $(\upp)$ of positive solutions to \eqref{problem} satisfying the condition
\begin{equation}
\label{energylimit}
p\int_{\Omega}|\nabla u_p|^2 dx\to\beta\in\R,\quad\mbox{as $p\to+\infty$}.
\end{equation}
Before stating our theorem let us review some known results. The first papers performing an asymptotic analysis of \eqref{problem}, as $p\to+\infty$, are \cite{RenWeiTAMS1994} and \cite{RenWeiPAMS1996} where the authors prove a $1$-point concentration phenomenon for least energy (hence positive) solutions to \eqref{problem} and derive some asymptotic estimates. Note that least energy solutions $(\upp)$ of the $2$-dimensional Lane-Emden problem satisfy the condition
\begin{equation}
\label{energylimit8pie}
p\int_{\Omega}|\nabla u_p|^2 dx\to8\pi e\in\R,\quad\mbox{as $p\to+\infty$}.
\end{equation}
which is a particular case of \eqref{energylimit}. Later, Adimurthi and Grossi (\cite{AdiGrossi}) identified a \emph{limit problem} by showing that suitable scalings of the least energy solutions $(\upp)$ converge in $C^2_{loc}(\R^2)$ to a regular solution $U$ of the Liouville problem
\begin{equation}\label{Liouville}
\left\{
\begin{array}{ll}
-\Delta U=e^U &\mbox{in $\R^2$}\\
\int_{\R^2}e^U dx=8\pi
\end{array}
\right.
\end{equation}
They also showed that $\|\upp\|_\infty$ converges to $\sqrt{e}$ as $p\to+\infty$, as it had been previously conjectured. Note that solutions to \eqref{problem}, satisfying \eqref{energylimit} do not blow up as $p\to+\infty$ (unlike the higher dimensional case when $p$ approaches the critical exponent). Concerning general positive solutions (i.e. not necessarily with least energy) a first asymptotic analysis was carried out in \cite{DeMarchisIanniPacellaJEMS} (see also \cite{DeMarchisIanniPacellaSurvey}) showing that, under the condition \eqref{energylimit}, all solutions $(\upp)$ concentrate at a finite number of points in $\bar\Omega$. Let us observe that this result holds even for sign changing solutions and can be also obtained substituting \eqref{energylimit} by a uniform bound on the Morse index of the solutions $(\upp)$ (see \cite{DeMarchisIanniPacellaAMPA}).\\
The results obtained in \cite{DeMarchisIanniPacellaJEMS} were then applied to the study of the asymptotic behavior of some families $(v_p)$ of sign changing solutions in symmetric domains to prove that suitable scalings of the positive parts $(v_p^+)$ converge to the function $U$ in \eqref{Liouville} while suitable scalings and translations of the negative parts $(v_p^-)$ converge to a singular radial solution of
\begin{equation}\label{LiouvilleSingolare}
\left\{
\begin{array}{ll}
-\Delta V=e^V+H\delta_0 &\mbox{in $\R^2$}\\
\int_{\R^2}e^V dx<+\infty
\end{array}
\right.
\end{equation}
where $H$ is a suitable constant and $\delta_0$ is the Dirac measure concentrated at the origin. A similar result had been previously shown in \cite{GrossiGrumiauPacella1} in the case of nodal radial solutions in the ball. Thus the limit profile of these sign-changing solutions looks like a superposition of two bubbles, one coming from the concentration of the positive parts and another coming from the concentration of the negative parts, both at the same point.\\
It is natural to ask weather a similar phenomenon appears when dealing with positive solutions. As a byproduct of our results we prove that this is not the case and in fact all concentration points of positive solutions are simple (and isolated) in the sense that there are not concentrating sequences converging at the same point. Moreover no scalings of positive solutions can converge to a singular solution of the Liouville problem in $\R^2$ and the concentration points are far away from the boundary of $\Omega$.\\
More precisely we prove:
\begin{theorem}
\label{teo:Positive}
Let $(u_p)$ be a family of positive solutions to \eqref{problem} satisfying
\eqref{energylimit}.\\
Then there exists a finite set of points $\mathcal S=\{x_1,\ldots,x_k\}\in\Omega$, $k\in\N\setminus\{0\}$ such that, up to sequences, $(u_p)$ satisfies the following properties:
\begin{itemize}
\item[$(i)$] $\sqrt{p}u_p\to0$ \ in $C^2_{loc}(\bar\Omega\setminus \mathcal S)$, \ as $p\to+\infty$;
\item[$(ii)$]
\[
pu_p(x)\to  8\pi \sum_{i=1}^k m_i G(x,x_i)\mbox{ \ in $C^2_{loc}(\bar\Omega\setminus \mathcal S)$ \ as $p\to+\infty$,}
\]
where $m_i:=\lim_{p\rightarrow +\infty} \|u_p\|_{L^{\infty}(\overline{B_{\delta}(x_i)})}$, for $\delta>0$ sufficiently small and $G$ is the Green's function of $-\Delta$ in $\Omega$ under Dirichlet boundary conditions;
\item[$(iii)$]
\begin{equation}\label{energylimitm_i^2}
p\int_\Omega |\nabla u_p(x)|^2\,dx\to 8\pi \sum_{i=1}^k m_i^2,\  \mbox{ as }p\to+\infty;
\end{equation}
\item[$(iv)$] the concentration points $x_i, \ i=1,\ldots, k$ satisfy
\begin{equation}\label{x_j relazione}
m_i \nabla_x H(x_i,x_i)+\sum_{\ell\neq i } m_\ell \nabla_x G(x_i,x_\ell)=0,
\end{equation}
where
\begin{equation}\label{H parteregolareGreen}
H(x,y)=G(x,y)+\frac{\log(|x-y|)}{2\pi}
\end{equation}
is the regular part of the Green's function $G$;
\item[$(v)$]
\[m_i \geq \sqrt{e},\qquad \forall \ i=1,\ldots,k.\]
\end{itemize}
\end{theorem}

Note that in particular we get
\[
\lim_{p\rightarrow +\infty} \|u_p\|_{\infty}\geq \sqrt{e}
\]
so that, by \eqref{energylimit} and \eqref{energylimitm_i^2} it follows:
\[
\beta=8\pi \sum_{i=1}^k m_i^2\geq k\ 8\pi e
\]
and hence the number of concentration points $k$ is estimated by:
\[ k\leq \left[\frac{\beta}{8\pi e}\right].\]

\

\begin{remark}
As observed before, for least energy solutions the limit \eqref{energylimit8pie} holds so that Theorem \ref{teo:Positive} implies that $k=1$, which was known from \cite{RenWeiPAMS1996},  and that  $\lim\limits_{p\to+\infty}\|u_p\|_{\infty}=\sqrt{e}$, which was already proved in \cite{AdiGrossi}.\\
We conjecture that for any family of positive solutions $u_p$ satisfying \eqref{energylimit} it should hold:
\begin{equation}\label{conjecture}
m_i=\sqrt{e},\quad \forall\, i=1,\ldots,k,
\end{equation}
and so in particular $\lim\limits_{p\to+\infty}\|u_p\|_{\infty}=\sqrt{e}$.\\
Note that if \eqref{conjecture} holds we would have a precise quantization of the energy which would imply that the limit energy level $\beta$ in \eqref{energylimit} is exactly:
\[
\beta=k\ 8\pi e, \;\, k\in\N\setminus\{0\}.
\]
Hence, for $p$ large, positive solutions $(\upp)$ to \eqref{problem} could exist only at levels of energy $p\int_\Omega |\nabla u_p|^2 dx$ close to a multiple of $8\pi e$. Therefore for positive solutions of Lane-Emden problems in dimension two the constant $8\pi e$ would play the same role as the Sobolev constant $S$ in dimension higher than or equal to $3$.
\\
We recall that concentrating positive solutions satisfying \eqref{conjecture} have been constructed in \cite{EspositoMussoPistoiaPos} for non simply connected  domains.
\end{remark}

The starting point to prove Theorem \ref{teo:Positive} is the asymptotic analysis performed in \cite{DeMarchisIanniPacellaJEMS} (see Section \ref{section:preliminaryresults}). Then the proof proceeds following some arguments used in \cite{SantraWei} to study the asymptotic behavior of solutions of biharmonic equations.

The outline of the paper is the following. In Section \ref{section:preliminaryresults} we recall preliminary results. In Section \ref{section:noconcentrationboundary} we show that the concentration points cannot belong to the boundary of $\Omega$. In Section \ref{section:scalinglocalmaxima} we analyze the rescaling around the local maxima of $u_p$. In Section \ref{section:conclusion} 
we %estimate the $L^{\infty}$-norm of the solutions and 
conclude the proof of Theorem \ref{teo:Positive}.

\

\

\

\

\section{Preliminary results}\label{section:preliminaryresults}

\

We start by recalling the classical Pohozaev identity.
\begin{lemma}[Pohozaev identity \cite{Pohozaev, PucciSerrin}] Let $A\subset \mathbb R^2$ be a smooth bounded domain, $u\in C^2(\bar A)$ a solution of $-\Delta u= f(u)$ and $F(u)=\int_0^u f(t)dt$. Then
\begin{eqnarray}
\label{Pohozaev}\nonumber
2\int_{A}F(u)dx &=& \int_{\partial A} F(u(x)) \langle x-y, \nu(x) \rangle ds_x
 -\frac{1}{2}\int_{\partial A}|\nabla u(x)|^2\langle x-y, \nu(x)\rangle ds_x
\\
&&
 + \int_{\partial A}\langle x-y, \nabla u(x)\rangle \langle \nabla u(x), \nu(x) \rangle ds_x
\end{eqnarray}
where $\nu(x)$ denotes the outer normal  at $\partial A$ at $x$, and $y\in \mathbb R^2$.
\end{lemma}

\

\

Assuming that $\partial\Omega\in C^2$, we consider the Green's function of $-\Delta$ on $\Omega$ under Dirichlet boundary conditions, namely the function $G$ which satisfies for any $y\in\Omega$
\begin{equation}
\label{eqGreenFunction}
\left\{
\begin{array}{lr}
-\Delta_x G(x,y)=\delta_y(x) & x\in\Omega
\\
G(x,y)=0 & x\in\partial\Omega
\end{array}
\right.
\end{equation}
where $\delta_y$ is the Dirac mass supported in $y$.

\

We denote by $H(x,y)$ the regular part of $G$, namely
\begin{equation}\label{defH}
H(x,y)=G(x,y)+\frac{1}{2\pi}\log|x-y|,
\end{equation}
which satisfies, for all $y\in\Omega$:
\begin{equation}
\label{eqH}
\left\{
\begin{array}{lr}
-\Delta_x H(x,y)=0 & x\in\Omega
\\
H(x,y)= \frac{1}{2\pi}\log|x-y| & x\in\partial\Omega.
\end{array}
\right.
\end{equation}

We recall that $H$ is a smooth function in $\Omega\times\Omega$, $G$ and $H$ are symmetric in $x$ and $y$. Moreover by the comparison principle
\begin{equation}
\label{boundH}
\frac{1}{2\pi}\log|x-y|< H(x,y)\leq C\quad\forall x,y\in\Omega,
\end{equation}
from which
\begin{equation}
\label{GMaggZero}
G(x,y)>0 \quad\forall x,y\in\Omega,
\end{equation}
and there exists $C_{\delta}>0$ such that
\begin{equation}\label{GlimitataFuoriDaPalla}
G(x,y)\leq C_{\delta}\qquad \forall\ |x-y|\geq \delta>0.
\end{equation}
One can also prove that (see for instance \cite[Lemma A.2]{Daprile}
\[\frac{\partial H}{\partial \nu_x}(x,y)=\frac{1}{2\pi} \frac{\partial }{\partial \nu_x}\left(\log |x-y|\right) +O(1), \quad \forall\, x\in\partial\Omega,\ \forall\, y\in\Omega\]
from which
\begin{equation}\label{GreenBoundedAtBoundary}
\left|\frac{\partial G}{\partial \nu_x}(x,y)\right|\leq C, \quad \forall \,x\in\partial\Omega,\ \forall\, y\in\Omega.
\end{equation}
Moreover (see for instance \cite{DallAcquaSweers}) one also has
\begin{equation}
\label{boundDerivataG}
|\nabla_x G(x,y)|\leq \frac{C}{|x-y|}\quad\forall x,y\in\Omega,\ x\neq y.
\end{equation}

\

\

Next we recall results already known about the asymptotic behavior of a general family $u_p$ of nontrivial solutions of \eqref{problem}, even sign-changing, satisfying the condition \eqref{energylimit}. This part is mainly based on some of the results contained in \cite{DeMarchisIanniPacellaJEMS}, plus smaller additions or minor improvements.  \\

\

In \cite{RenWeiTAMS1994}  it has been proved that for any family
$(u_p)_{p>1}$ of nontrivial solutions of \eqref{problem}  the following lower bound holds
\begin{equation}
\label{energylimitLower}
\liminf_{p\rightarrow+\infty}p\int_{\Omega}|\nabla u_p|^2 dx\geq 8\pi e,
\end{equation}
which implies that the constant $\beta$ in \eqref{energylimit} satisfies $\beta\geq 8\pi e$.
\

If we denote by $E_p$ the energy functional associated to \eqref{problem}, i.e.
\[
E_p(u):=\frac{1}{2}\|\nabla u\|^2_{2}-\frac{1}{p+1}\|u\|_{p+1}^{p+1},\ \ u\in H^1_0(\Omega),
\]
since for a solution $u$  of \eqref{problem}
\begin{equation}
\label{energiaSuSoluzioni}
E_p(u)=(\frac12-\frac1{p+1})\|\nabla u\|^2_2=(\frac12-\frac1{p+1})\|u\|^{p+1}_{p+1},
\end{equation}
then \eqref{energylimit} and \eqref{energylimitLower} are equivalent to lower bounds for the limit of the energy $E_p$ or for the $L^{p+1}$-norm, namely
\[
\displaystyle\lim_{p\rightarrow +\infty}\  2 pE_p(u_p) = \lim_{p\rightarrow +\infty}\  p\int_{\Omega} |u_p|^{p+1} \, dx = \lim_{p\rightarrow +\infty}\  p\int_{\Omega} |\nabla u_p|^{2} \, dx=\beta\geq 8\pi e
\]
we will use all these equivalent formulations in the sequel.
\

\

Observe that by the assumption in \eqref{energylimit} we have that
\[E_p(u_p)\rightarrow 0, \ \ \|\nabla u_p\|_2\rightarrow 0,\ \  \mbox{as $p\rightarrow +\infty$}\]
so in particular $u_p\rightarrow 0$ a.e. as $p\rightarrow +\infty$.
\\
On the other side it is known that the solutions $u_p$ do not vanish as $p\rightarrow +\infty$ and that they do not blow-up, unlike the higher dimensional case. Indeed the following results hold:
\begin{proposition}\label{prop:BoundEnergia}
Let $(\upp)$ be a family of solutions to \eqref{problem} satisfying \eqref{energylimit}. Then
\begin{itemize}

\item[\emph{$(i)$}] (No vanishing).\\
\[\|u_p\|_{\infty}^{p-1}\geq \lambda_1,\] where $\lambda_1=\lambda_1(\Omega)(>0)$ is the first eigenvalue of the operator $-\Delta$ in $H^1_0(\Omega)$.
\item[$(ii)$] (Existence of the first bubble).

Let $(x_p^+)_p\subset\Omega$ such that $u_p(x_p^+)=\|u_p\|_{\infty}$.
Let us set
\begin{equation}\label{muppiu}
\mu_{p}^+:=\left(p |\upp(x_p^+)|^{p-1}\right)^{-\frac{1}{2}}
\end{equation}
and  for  $ x\in\widetilde{\Omega}_{p}^+:=\{x\in\R^2\,:\,x_p^++\mu_{p}^+x\in\Omega\}$
\begin{equation}\label{scalingMax}
v_{p}^+(x):=\fr{p}{\upp(x_p^+)}(\upp(x_p^++\mu_{p}^+ x)-\upp(x_{1,p})).
\end{equation}
Then  $\mu_{p}^+\to0$ as $p\to+\infty$  and
\[v_{p}^+\To U\mbox{ in }C^2_{loc}(\R^2)\mbox{ as }p\to+\infty\] where
\begin{equation}\label{v0}
U(x)=\log\left(\fr1{1+\fr18 |x|^2}\right)^2
\end{equation}
is the solution of $-\lap U=e^{U}$ in $\R^2$, $U\leq 0$, $U(0)=0$ and $\int_{\mathbb{R}^2}e^{U}=8\pi$.
\\
Moreover
\begin{equation}\label{soloMaggiore1InGenerale}
\liminf_{p\rightarrow +\infty}\|u_p\|_{\infty}\geq 1.
\end{equation}

\item[\emph{$(iii)$}] (No blow-up). There exists $C>0$ such that
\begin{equation}
\label{boundSoluzio}\|\upp\|_{L^\infty(\Omega)}\leq C,\ \mbox{  for all $p>1$.}
\end{equation}
\item[\emph{$(iv)$}] There exist constants $c,C>0$, such that for all $p$ sufficiently large we have
\begin{equation}
\label{boundEnergiap}c\leq p\int_\Omega |u_p(x)|^p dx \leq C.
\end{equation}
\item[\emph{$(v)$}] $\sqrt{p}u_p\rightharpoonup 0$ in $H^1_{0}(\Omega)$ as $p\rightarrow +\infty$.
\end{itemize}
\end{proposition}
\begin{proof}
The statements \emph{$(i)$} and \emph{$(iii)$} have been first proved for positive solutions in \cite{RenWeiTAMS1994}, while \emph{$(ii)$} is essentially proved in \cite{AdiGrossi} (see also \cite{DeMarchisIanniPacellaSurvey}). Assertion \emph{$(iv)$} follows easily from \emph{$(iii)$}, by H\"older inequality and \eqref{energylimitLower} and \eqref{energiaSuSoluzioni}. The proof of \emph{$(v)$} is given in \cite{DeMarchisIanniPacellaJEMS} or \cite{DeMarchisIanniPacellaSurvey}.
\end{proof}

\

\

We now recall an important result about the asymptotic behavior of solutions to \eqref{problem} satisfying \eqref{energylimit} which has been proved in \cite{DeMarchisIanniPacellaJEMS}. It is the starting point for the proof of Theorem \ref{teo:Positive}.

In order to state it (see Proposition \ref{thm:x1N} below) we need to introduce some notations. Given a family $(u_p)$ of solutions of \eqref{problem}
and assuming that there exists  $n\in\N\setminus\{0\}$ families of points $(\xip)$, $i=1,\ldots,n$  in $\Omega$ such that
\begin{equation}
\label{muVaAZero}
p|\upp(\xip)|^{p-1}\to+\infty\ \mbox{ as }\ p\to+\infty,
\end{equation}
we define the parameters $\mip$ by
\bel\label{mip}
\mip^{-2}=p |\upp(\xip)|^{p-1},\ \mbox{ for all }\ i=1,\ldots,n.
\eel
By \eqref{muVaAZero} it is clear that $\mip\to0$ as $p\to+\infty$ and that
\begin{equation}\label{RemarkMaxCirca1}
\liminf_{p\rightarrow +\infty}|u_p(\xip)|\geq 1.
\end{equation}
Then we define the concentration set
\bel\label{S}
\mathcal{S}=\left\{\lim_{p\to+\infty}\xip,\,i=1,\ldots,n\right\}\subset\bar\Omega
\eel
and the function
\bel\label{RNp}
R_{n,p}(x)=\min_{i=1,\ldots,n} |x-\xip|, \ \forall x\in\Omega.
\eel

\

Finally we introduce the following properties:

\

\begin{itemize}
\item[$(\mathcal{P}_1^n)$] For any $i,j\in\{1,\ldots,n\}$, $i\neq j$,
\[
\lim_{p\to+\infty}\fr{|\xip-\xjp|}{\mip}=+\infty.
\]
\item[$(\mathcal{P}_2^n)$] For any $i=1,\ldots,n$, for $x\in\widetilde{\Omega}_{i,p}:=\{x\in\R^2\,:\,x_{i,p}+\mu_{i,p}x\in\Omega\}$
\begin{equation}\label{vip}
v_{i,p}(x):=\fr{p}{\upp(\xip)}(\upp(\xip+\mip x)-\upp(\xip))\To U(x)
\end{equation}
in $C^2_{loc}(\R^2)$ as $p\to+\infty$, where $U$ is the same function in \eqref{v0}.
\item[$(\mathcal{P}_3^n)$] There exists $C>0$ such that
\[
p R_{n,p}(x)^2 |\upp(x)|^{p-1}\leq C
\]
for all $p>1$ and all $x\in \Omega$.
\item[$(\mathcal{P}_4^n)$] There exists $C>0$ such that
\[
p R_{n,p}(x) |\nabla \upp(x)|\leq C
\]
for all $p>1$ and all $x\in \Omega$.
\end{itemize}

\

\begin{lemma}\emph{(}\cite[Lemma 2.1 - \emph{$(iii)$}]{DeMarchisIanniPacellaJEMS}\emph{)}\label{lemma:BoundEnergiaBassino}
If there exists  $n\in\N\setminus\{0\}$ such that the properties $(\mathcal{P}_1^n)$ and $(\mathcal{P}_2^n)$ hold for families $(\xip)_{i=1,\ldots,n}$ of points satisfying \eqref{muVaAZero}, then
\[
p\int_\Omega |\na\upp|^2\,dx\geq8\pi\sum_{i=1}^n \alpha_i^2+o_p(1)\ \mbox{ as }p\rightarrow +\infty,
\]
where  $\alpha_i:=\liminf_{p\to+\infty}|\upp(\xip)|\ (\overset{\eqref{RemarkMaxCirca1}}{\geq} 1)$.

\end{lemma}

\

Next result shows that the solutions concentrate at a finite number of points and also establishes the existence of a maximal number of ``bubbles''

\begin{proposition}\emph{(}\cite[Proposition 2.2]{DeMarchisIanniPacellaJEMS}, \cite[Theorem 2.3]{DeMarchisIanniPacellaSurvey}\emph{)}\label{thm:x1N}
Let $(\upp)$ be a family of solutions to \eqref{problem} and assume that \eqref{energylimit} holds. Then there exist $k\in\N\setminus\{0\}$ and $k$ families of points $(\xip)$ in $\Omega$  $i=1,\ldots, k$ such that, after passing to a sequence, $(\mathcal{P}_1^k)$, $(\mathcal{P}_2^k)$, and $(\mathcal{P}_3^k)$ hold. Moreover $x_{1,p}=x_p^+$ and, given any family of points $x_{k+1,p}$, it is impossible to extract a new sequence from the previous one such that $(\mathcal{P}_1^{k+1})$, $(\mathcal{P}_2^{k+1})$, and $(\mathcal{P}_3^{k+1})$ hold with the sequences $(\xip)$, $i=1,\ldots,k+1$. At last, we have
\begin{equation}\label{pu_va_a_zero}
\sqrt{p}\upp\to 0\quad\textrm{ in $C^2_{loc}(\bar\Omega\setminus\mathcal{S})$ as $p\to+\infty$.}
\end{equation}
Moreover there exists $v\in C^2(\bar\Omega\setminus\mathcal{S})$ such that
\begin{equation}
\label{pu_va_a_funzione}
p\upp\to v\quad\textrm{ in $C^2_{loc}(\bar\Omega\setminus\mathcal{S})$ as $p\to+\infty$,}
\end{equation}
and $(\mathcal{P}_4^k)$ holds.
\end{proposition}

\

In the rest of this section we derive some consequences of Proposition \ref{thm:x1N}.

\

\begin{remark}\label{rem:nonvedobordo}
Under the assumptions of Proposition \ref{thm:x1N} we have
$$
\fr{dist(\xip,\partial\Omega)}{\mip}\underset{p\to+\infty}{\longrightarrow}+\infty\qquad\textrm{for all $i\in\{1,\ldots,k\}$}.
$$
\end{remark}

\

\begin{corollary}
Let $K\subset\bar\Omega\setminus\mathcal{S}$ be a compact set. Then
\begin{equation}\label{pupVaAZeroUnifSuiCompatti}
\lim_{p\rightarrow +\infty} \|p|u_p(x)|^{p}\|_{L^{\infty}(K)}=\lim_p \|p|u_p|^{p+1}\|_{L^{\infty}(K)}=0
\end{equation}
and so
\begin{equation}
\label{integraleVaAZeroSuCompatti}
\lim_{p\rightarrow +\infty} p\int_K |u_p(x)|^p dx= \lim_{p\rightarrow +\infty} p\int_K |u_p(x)|^{p+1} dx=0.
\end{equation}
Moreover
\begin{equation}\label{pGradupVaAZeroUnifSuiCompatti}
\lim_{p\rightarrow +\infty} \|p|\nabla u_p(x)|^{2}\|_{L^{\infty}(K)}=0
\end{equation}
and so
\begin{equation}
\label{integraleGradVaAZeroSuCompatti}
\lim_{p\rightarrow +\infty} p\int_K |\nabla u_p(x)|^2 dx=0.
\end{equation}
\end{corollary}
\begin{proof}
If $K$ is a compact subset of $\bar\Omega\setminus\mathcal{S}$ by $(\mathcal{P}_3^k)$ we have that there exists $C_K>0$ such that
\begin{equation} \label{PkinCompatto}
p|u_p(x)|^{p-1}\leq C_K, \ \mbox{ for all }x\in K.
\end{equation}
As a consequence for $x\in K$
\begin{equation}\label{pu^pVaAZeroUnifSuCompatti}
p|u_p(x)|^{p+1}\leq \|u_p\|_{\infty} p|u_p(x)|^{p}\overset{\mbox{\scriptsize Proposition \ref{prop:BoundEnergia}-(iii)}}{\leq} C\, p|u_p(x)|^{p}  \overset{\eqref{PkinCompatto}}{\leq} C_K\,u_p(x)      \rightarrow 0
\end{equation} uniformly as $p\rightarrow +\infty$ by \eqref{pu_va_a_zero}.

\

The proof of \eqref{pGradupVaAZeroUnifSuiCompatti} follows similarly by using $(\mathcal{P}_4^k)$ instead of $(\mathcal{P}_3^k)$.
\end{proof}

\

\

For a family of points $(x_p)_p\subset\Omega$ we denote by $\mu(x_p)$ the numbers defined by
\begin{equation}
\label{defmup}
\left[\mu (x_p)\right]^{-2}:=p |u_p(x_p)|^{p-1}.
\end{equation}

\begin{proposition}\emph{(}\cite[Proposition 2.5]{DeMarchisIanniPacellaJEMS}\emph{)}\label{lemma:rapportoMuBounded} Let $(x_p)_p\subset\Omega$ be a family of points such that  $p |u_p(x_p)|^{p-1}\rightarrow +\infty$ and let $\mu (x_p)$ be as in \eqref{defmup}.
Let $i\in\{1,\ldots,k\}$ such that $R_{k,p}(x_p)=|x_{i,p}-x_p|$, up to a sequence, then
\[
\limsup_{p\rightarrow +\infty}\frac{\mu_{i,p}}{\mu (x_p)}\leq  1.
\]
\end{proposition}

\

\

Next result characterizes in different ways the concentration set $\mathcal S$.

\begin{proposition}\emph{(}\cite[Proposition 2.9]{DeMarchisIanniPacellaSurvey}\emph{)}\label{prop:stime gradiente}%\label{propozition: caratterizzazioneS}
Let $(u_p)$ be a family of solutions to \eqref{problem} satisfying \eqref{energylimit}. Then the following holds:
\begin{itemize}
\item[\emph{$(i)$}]
\[\mathcal S=\left\{x\in\overline\Omega\, :\, \forall\, r_0>0,\  \forall\, p_0>1,\ \exists\, p>p_0\  \mbox{ s.t. }\ p\int_{B_{r_0}(x)\cap\Omega} |u_p(x)|^{p+1}\ dx\geq 1\right\};\]
\item[\emph{$(ii)$}]
\[
\mathcal{S}=\left\{x\in\overline\Omega\, :\,
\begin{array}{l}
\exists\, \mbox{a sequence $(u_{p_n})\subset(u_p)$ and a sequence of points $x_{p_n}\to x$,}\\
\mbox{s.t. $p_n|u_{p_n}(x_{p_n})|\to+\infty$ as $p_n\to+\infty$}
\end{array}
\right\}.
\]
\end{itemize}
\end{proposition}

\

\

\

\section{No concentration at the boundary}\label{section:noconcentrationboundary}

\

Let $k\in\N\setminus\{0\}$ be as in  Proposition \ref{thm:x1N} the maximal number of families of points $(\xip)\subset \Omega$,  $i=1,\ldots, k$ which up to a sequence satisfy $(\mathcal{P}_1^k)$, $(\mathcal{P}_2^k)$, and $(\mathcal{P}_3^k)$. In \eqref{S} we have also defined
\[\mathcal S=\left\{\lim_{p\rightarrow +\infty}\xip, \ i=1,\ldots,k  \right\}\subset\overline\Omega,\]
for which the characterization in Proposition \ref{prop:stime gradiente} holds.

We denote by $N\in \N\setminus\{0\}$ the number of points in $\mathcal{S}$. Hence $N\leq k$, moreover w.l.g. we can relabel the sequences of points $x_{i,p}$, $i=1,\ldots, k$ and assume that
\begin{equation}\label{ribattezzo}
x_{j,p}\rightarrow x_j,\ \forall j=1,\ldots, N\quad\mbox{ and }\quad\mathcal{S}=\{ x_1, \ x_2, \ \ldots, \ x_N \}
\end{equation}

\

\

\begin{lemma}\label{lemma:sommeDiGreen}
There exists $\gamma_j>0$, $j=1,\ldots, N$ such that
\[\lim_{p\rightarrow +\infty}pu_p=\sum_{j=1}^N\ \gamma_j G(\cdot, x_j)\quad\mbox{ in }\ C^2_{loc}(\bar\Omega\setminus\mathcal S).\]
Moreover
\begin{equation}
\label{defgammaj}
\gamma_j=\lim_{\delta\rightarrow 0}\lim_{p\rightarrow +\infty}p\int_{B_{\delta}(x_j)\cap\Omega}u_p(x)^p\, dx,
\end{equation}
where $B_\delta(x_j)$ is a ball of center at $x_j$ and radius $\delta>0$.
\end{lemma}

\begin{proof}
Since the $x_j$'s are isolated, there exists $r>0$ such that $B_r(x_i)\cap B_r(x_j)=\emptyset$. Let $\delta\in (0,r)$, then by the Green representation formula
\begin{eqnarray*}
p\upp(x)&=& p\int_{\Omega}G(x,y)\upp(y)^{p}\,dy\\
&=& p\sum_{j=1}^N\int_{B_{\delta}(x_j)\cap\Omega}G(x,y)\upp(y)^{p}\,dy+\,p\int_{\Omega\setminus\cup_{j} B_{\delta}(x_j)}G(x,y)\upp(y)^{p}\,dy\\
&\overset{\eqref{GMaggZero}-\eqref{GlimitataFuoriDaPalla}-\eqref{integraleVaAZeroSuCompatti}}{=}&
p\sum_{j=1}^N\int_{B_{\delta}(x_j)\cap\Omega}G(x,y)\upp(y)^{p}\,dy +
o_p(1),
\end{eqnarray*}
Furthermore by the continuity of $G(x,\cdot)$ in $\bar\Omega	\setminus\{x\}$ and by Proposition \ref{prop:BoundEnergia}-\emph{$(iv)$} we obtain
\[
\lim_{p\to+\infty}p\upp(x)=\sum_{j=1}^N\gamma_j G(x,x_j),
\]
where
\[\gamma_j=\lim_{\delta\rightarrow 0}\lim_{p\rightarrow +\infty}p\int_{B_{\delta}(x_j)\cap\Omega}u_p(x)^{p}\, dx.\]

Last we show that $\gamma_j>0$.\\
Since $x_{j,p}\rightarrow x_j$ as $p\rightarrow +\infty$ then $B_{\frac{\delta}{2}}(x_{j,p})\subset B_{\delta}(x_{j})$  for $p$ large  and
so, since $u_p >0$
\[p\int_{B_{\delta}(x_j)\cap\Omega}u_p(x)^{p}\, dx\geq
p\int_{B_{\frac{\delta}{2}}(x_{j,p})\cap\Omega}u_p(x)^{p}\, dx= u_p(x_{j,p})\int_{B_{\frac{\delta}{2\mjp}}(0)\cap\widetilde\Omega_{j,p}} \left(1+\frac{v_{j,p}(x)}{p}\right)^{p}\, dx,\]
where the last equality is obtained by scaling around $\xjp$, where $v_{j,p}$ are defined in \eqref{vip} and $\widetilde\Omega_{j,p}=\{x\in\R^2\,:\,x_{i,p}+\mu_{i,p}x\in\Omega\}$.
Passing to the limit as $p\rightarrow +\infty$, since $B_{\frac{\delta}{2\mjp}}(0)\cap\widetilde\Omega_{j,p}\rightarrow \R^2$ and  $(\mathcal{P}_2^N)$ holds, by  Fatou's Lemma we get
\[\lim_p p\int_{B_{\delta}(x_j)\cap\Omega}u_p(x)^{p}\, dx\geq \liminf_p u_p(x_{j,p})\int_{\R^2}  e^{U(x)}\, dx\geq 8\pi>0,\]
having used  that $\liminf_p u_p(x_{j,p})\geq 1$ (see \eqref{RemarkMaxCirca1}).
\end{proof}

\

\

Next we show that there is no boundary blow-up
\begin{proposition}\label{proposition:noBoundary}
\[\mathcal S\cap \partial\Omega=\emptyset.\]
\end{proposition}

\begin{proof}
We argue by contradiction. Suppose that $x_i\in\mathcal S\cap\partial\Omega$, for some $i=\{1,\ldots,N\}$. Choose $r>0$ such that $\mathcal{S}\cap B_{r}(x_i)=\{x_i\}.$ By the characterization of $\mathcal S$ in Proposition \ref{prop:stime gradiente}-(i), we have that for all ${\delta}<r$, for all $p_0>1$  there exists $p>p_0$ such that
\begin{equation}\label{assurdo}
p\int_{\Omega\cap B_{\delta}(x_i)}u_p(x)^{p+1}\, dx\geq 1.
\end{equation}

Let $y_p:=x_i+\rho_{p,{\delta}}\nu(x_i)$, where
\[\rho_{p,{\delta}}:=\frac{\int_{\partial\Omega\cap B_{\delta}(x_i)}\left( \frac{\partial u_p(x)}{\partial \nu}  \right)^2 \langle x-x_i,\nu(x)\rangle ds_x}{\int_{\partial\Omega\cap B_{\delta}(x_i)}\left( \frac{\partial u_p(x)}{\partial \nu}  \right)^2\langle \nu(x_i),\nu(x)\rangle ds_x}\]
and ${\delta}<<r$ such that $\frac{1}{2}\leq \langle \nu(x_i),\nu(x)\rangle \leq 1$, for $x\in\partial\Omega\cap B_{\delta}(x_i)$.
With this choice of ${\delta}$ we have
\begin{equation}\label{stimaRo}|\rho_{p,{\delta}}|\leq 2{\delta}.\end{equation}
Moreover it is easy to see that the choice of $y_p$ implies
\begin{equation}
\label{primoPezzoBordo}
\int_{\partial\Omega\cap B_{\delta}(x_i)}   \left( \frac{\partial u_p(x)}{\partial \nu}  \right)^2  \langle x-y_p,\nu(x)\rangle ds_x
= 0.
\end{equation}

Applying the local Pohozaev identity \eqref{Pohozaev} in the set $\Omega\cap B_{\delta}(x_i)$ with $y=y_p$, using \eqref{primoPezzoBordo}, the boundary condition $u_p=0$ on $\partial\Omega$ (so that $\left|\frac{\partial u_p}{\partial \nu}\right|=|\nabla u_p| $ on $\partial\Omega$) we obtain

\begin{eqnarray}\label{pohoqui}
\frac{2p^2}{p+1}\int_{\Omega\cap B_{\delta}(x_i)}u_p(x)^{p+1}dx
& = &
\frac{p^2}{2}\int_{\partial \Omega\cap  B_{\delta}(x_i)}|\nabla u_p(x)|^2\langle x-y_p, \nu(x)\rangle ds_x
\nonumber\\
&&\ - \
\frac{1}{2}\int_{\Omega\cap \partial B_{\delta}(x_i)}|p\nabla u_p(x)|^2\langle x-y_p, \nu(x)\rangle ds_x
\nonumber\\
&&
\ + \ \int_{\Omega\cap \partial B_{\delta}(x_i)}\langle x-y_p, p \nabla u_p(x)\rangle \langle p\nabla u_p(x), \nu(x)\rangle ds_x
 \nonumber\\
&&\ + \
\frac{p^2}{(p+1)} \int_{\Omega\cap \partial B_{\delta}(x_i)}u_p(x)^{p+1}  \langle x-y_p,\nu(x)\rangle ds_x
\nonumber
\\
& \overset{\eqref{primoPezzoBordo}}{=} &
\ - \
\frac{1}{2}\int_{\Omega\cap \partial B_{\delta}(x_i)}|p\nabla u_p(x)|^2\langle x-y_p, \nu(x)\rangle ds_x
\nonumber\\
&&
\ + \ \int_{\Omega\cap \partial B_{\delta}(x_i)}\langle x-y_p, p\nabla u_p(x)\rangle \langle p \nabla u_p(x), \nu(x)\rangle ds_x
 \nonumber\\
&&\ + \
\frac{p^2}{(p+1)} \int_{\Omega\cap \partial B_{\delta}(x_i)}u_p(x)^{p+1}  \langle x-y_p,\nu(x)\rangle ds_x
\end{eqnarray}
Next we show that the three terms in the right hand side are $O({\delta}^2)$.

By Lemma \ref{lemma:sommeDiGreen} we have in particular that $pu_p\rightarrow \sum_{j=1}^N\gamma_j G(\cdot,x_j)$ in $C^2_{loc}(\overline{\Omega}\cap B_{r}(x_i)\setminus\{x_i\})$. Hence it is easy to see that for $x\in \overline{\Omega}\cap B_{r}(x_i)\setminus\{x_i\}$, since $x_i\in\partial\Omega$:
\begin{eqnarray}
\sum_{j=1}^N\gamma_j G(x,x_j)&=&\gamma_iG(x,x_i) + O(1)\overset{ \eqref{eqGreenFunction}}{=} O(1)
\nonumber
\\
\sum_{j=1}^N\gamma_j \nabla_x G(x,x_j)&=&\gamma_i \nabla_x G(x,x_i) + O(1)\overset{G \ \mbox{\scriptsize is symm}}{=} \gamma_i \nabla_y G(x,y)|_{y=x_i} + O(1)
\nonumber\\
&\overset{G\equiv 0\ \mbox{\scriptsize on }\partial\Omega}=&
\gamma_i \frac{\partial G(x,x_i)}{\partial\nu} + O(1)
\overset{\eqref{GreenBoundedAtBoundary}}{=} O(1).
%\nonumber
%\\
%&&\sum_{j=1}^N\gamma_j \frac{\partial G(x,x_j)}{\partial\nu}=\gamma_i \frac{\partial G(x,x_i)}{\partial\nu} + O(1)\overset{\eqref{GreenBoundedAtBoundary}}{=} O(1).
\label{nablapupgr}
\end{eqnarray}

So, by the uniform convergence of $pu_p$ and its derivative on compact sets, it follows  that
\begin{eqnarray*}
&& -
\frac{1}{2}\int_{\Omega\cap \partial B_{\delta}(x_i)}|p\nabla u_p(x)|^2\langle x-y_p, \nu(x)\rangle ds_x
\overset{\eqref{nablapupgr}}{=}O(1) \int_{\Omega\cap \partial B_{\delta}(x_i)}\langle x-y_p,\nu(x)\rangle ds_x =O({\delta}^2)
\\
&&
\int_{\Omega\cap \partial B_{\delta}(x_i)}\langle x-y_p, p\nabla u_p(x)\rangle \langle p \nabla u_p(x), \nu(x)\rangle ds_x
\overset{\eqref{nablapupgr}}{=}O(1) \int_{\Omega\cap \partial B_{\delta}(x_i)}|x-y_p| ds_x =O({\delta}^2)
\end{eqnarray*}
\begin{eqnarray*}
&&\frac{p^2}{p+1} \left|\int_{\Omega\cap \partial B_{\delta}(x_i)}u_p(x)^{p+1}  \langle x-y_p,\nu(x)\rangle ds_x\right| =
\qquad\qquad  \\
&& \qquad\qquad\qquad\qquad
\leq\frac{p}{p+1}  \|p u_p^{p+1}\|_{L^{\infty}(\bar\Omega\cap\partial B_{\delta}(x_i))}\int_{\Omega\cap \partial B_{\delta}(x_i)}\left|\langle x-y_p,\nu(x)\rangle \right|ds_x\overset{\eqref{pupVaAZeroUnifSuiCompatti}}{=} o_p(1)O({\delta}^2)
\end{eqnarray*}
where in all the three cases for the last equality we have also used the estimate in  \eqref{stimaRo}.
As a consequence, by \eqref{pohoqui},
\[\lim_{{\delta}\rightarrow 0}\lim_{p\rightarrow +\infty}\frac{p^2}{p+1}\int_{\Omega\cap B_{\delta}(x_i)}u_p(x)^{p+1}dx=0,\]
a contradiction to \eqref{assurdo}.
\end{proof}

\

\

\

\section{Scaling around local maxima}\label{section:scalinglocalmaxima}

\

By Proposition \ref{proposition:noBoundary} it follows that there exists $r>0$ such that
\begin{equation}\label{rPiccoloAbbastanza}
B_{4r}(x_i)\cap B_{4r}(x_j)=\emptyset,\quad B_{4r}(x_i)\subset \Omega,\quad \mbox{for all $i=1,\ldots, N$, $i\neq j$.}
\end{equation}

\

\begin{lemma}\label{lemma:possoRiscalareAttornoAiMax}
Let $N\in\mathbb N\setminus\{0\}$ be as in \eqref{ribattezzo} and let $r>0$ be as in \eqref{rPiccoloAbbastanza}. Let us define $y_{j,p}\in \overline{B_{2r}(x_j)}$, $j=1,\ldots, N$  such that \begin{equation}\label{defy}u_p(y_{j,p})=\max_{\overline{B_{2r}(x_j)}}u_p(x).
\end{equation}
Then, for any $j=1,\ldots, N$ and as $p\to+\infty$:
\begin{itemize}
\item[(i)]\[
\varepsilon_{j,p}:=\left[ p \upp(y_{j,p})^{p-1}\right]^{-1/2}\to 0.\]

\item[(ii)]  \[y_{j,p}\rightarrow x_j.\]

\item[(iii)]
\[
\fr{|y_{i,p}-y_{j,p}|}{\varepsilon_{j,p}}\to +\infty\ \mbox{for any $i=1,\ldots,N$, $i\neq j$.}
\]
\item[(iv)] Defining:
\begin{equation}\label{defRiscalataMax}w_{j,p}(y):=\fr{p}{\upp(y_{j,p})}(\upp(y_{j,p}+\varepsilon_{j,p} y)-\upp(y_{j,p})),\quad  y\in  \Omega_{j,p}:=\frac{\Omega-y_{j,p}}{\varepsilon_{j,p}},
%
%   B_{\frac{r}{\varepsilon_{j,p}}}\left(\frac{x_j-y_{j,p}}{\varepsilon_{j,p}} \right),
\end{equation}
then
\begin{equation}\label{convRiscalateNeiMax}
w_{j,p}\To U\ \mbox{ in }\ C^2_{loc}(\R^2)
\end{equation}
with $U$ as in \eqref{v0}.

\item[(v)]
\begin{equation}
\label{energiaBPallaLontanaDaZero}
\liminf_pp\int_{B_{r}(x_j)}u_p(x)^{p+1}\, dx\geq 8\pi\cdot \liminf_p u_p(y_{j,p})^2\ (>8\pi).
\end{equation}
\item[(vi)]\[\liminf_pp\int_{B_{r}(y_{j,p})}u_p(x)^{p}\, dx\geq 8\pi\cdot \liminf_p u_p(y_{j,p}).\]
\end{itemize}
\end{lemma}

\begin{remark} $(iii)$ and $(iv)$ are respectively properties  $(\mathcal{P}_1^N)$ and $(\mathcal{P}_2^N)$ for the families of points $y_{j,p}$, $j=1,\ldots, N$.
Moreover by $(i)$ we get
\begin{equation}
\label{maxLocMag1}
\liminf_{p\rightarrow +\infty} u_p(y_{j,p})\geq 1
\end{equation}
and by (ii) we deduce that for any $\delta\in(0,2r)$ there exists $p_{\delta}>1$ such that
\begin{equation}\label{massimoVaAfinireInPallaPiccolaApiacere}
y_{j,p}\in B_{\delta} (x_j),\quad\mbox{ for } p\geq p_{\delta}.
\end{equation}
\end{remark}

\begin{proof}

(i): let $j\in\{1,\ldots,N\}$, by \eqref{ribattezzo}   $\xjp\rightarrow x_j$ as $p\rightarrow +\infty$ and so $\xjp\in B_r(x_j)$ for $p$ large. The assertion then follows observing that by definition $u_p(y_{j,p})\geq u_p(\xjp)$ and that \eqref{muVaAZero} holds for $\xjp$.

\

(ii): we know that $\xjp\rightarrow x_j$ as $p\rightarrow +\infty$, so w.l.g.  we may assume that
$R_{k,p}(y_{j,p})=|x_{j,p}-y_{j,p}|$ and so
$(\mathcal{P}_3^k)$ may be written as $\frac{|\xjp-y_{j,p}|}{\varepsilon_{j,p}}\leq C$, from which by (i) the conclusion follows.

\

(iii): just observing that by construction $|y_{i,p}-y_{j,p}|\geq 6r$ if $i\neq j$.

\

(iv): First we prove that for any $R>0$ there exists   $p_R>1$ such that
\begin{equation}\label{inclusioniSets}
B_R(0)\subset B_{\frac{2r}{\varepsilon_{j,p}}}\left(\frac{x_j-y_{j,p}}{\varepsilon_{j,p}} \right)\subset \Omega_{j,p}\ \mbox{
 for }\ p\geq p_R.
 \end{equation}
Indeed using (ii) and (i) we get respectively that $y_{j,p}\in B_{r}(x_j)$ and  $R\varepsilon_{j,p}<r$  for $p$ large. As a consequence $ B_{R\varepsilon_{j,p}}(y_{j,p})\subset B_{2r}(x_j)\subset\Omega$ for $p$ large, which gives  \eqref{inclusioniSets} by scaling back.

\

Observe that by \eqref{inclusioniSets} and the arbitrariness of $R$ it follows that
the set $\Omega_{j,p}
  \rightarrow \mathbb R^2$
as $p\rightarrow +\infty$.

Moreover let us fix $R>0$ and let $p_R$ be as in \eqref{inclusioniSets}, then for $p\geq p_R$ the function $w_{j,p}$ satisfies
\[
\left\{
\begin{array}{lr}
-\Delta w_{j,p}(y) = \left(\frac{\upp(y_{j,p}+\varepsilon_{j,p} y)}{(\upp(y_{j,p})}\right)^{p},  & y\in B_R(0)
\\
w_{j,p}(0) = 0
\end{array}
\right.
\]
and by the first inclusion in  \eqref{inclusioniSets} and the definition of $y_{j,p}$ we have
\[
u_p(y_{j,p}+\varepsilon_{j,p} y)\leq u_p(y_{j,p}),\ \mbox{ for any $y\in B_R(0)$, \ for $p\geq p_R$. }
\]
This implies both
\begin{equation}\label{minZero}w_{j,p}(y) \leq 0,  \  y\in  B_R(0)\end{equation}
and
\begin{equation}\label{bddLapl}|-\Delta w_{j,p}(y) |\leq 1,  \  y\in  B_R(0),\end{equation} for $p\geq p_R$.
From \eqref{minZero} and \eqref{bddLapl}, arguing as in the proof of  Proposition \ref{thm:x1N}, it follows that, for any $R>0$,  $w_{j,p}$ is uniformly bounded in $B_R(0)$, for $p\geq p_R$.
By standard elliptic regularity theory we have that $w_{j,p}$ is bounded in $C^{2,\alpha}_{loc}(\R^2).$
Thus by Arzela-Ascoli Theorem and a diagonal process on $R\rightarrow +\infty$, after passing to a subsequence,
 $w_{j,p}\rightarrow_p v \mbox{ in }\ C^2_{loc}(\R^2)$ and it is easy to see that  $v$ satisfies
\[
\left\{
\begin{array}{lr}
-\Delta v = e^{v}  & \mbox{in }\R^2
\\
v\leq 0  & \mbox{in }\R^2\\
v(0) = 0\\
\int_{\R^2}e^v<\infty
\end{array}
\right.
\]
hence  $v=U$ where $U$ is the function in \eqref{v0}.
\\
\\
(v) and (vi):
using \eqref{massimoVaAfinireInPallaPiccolaApiacere} we have that $y_{j,p}\in B_{\frac{r}{2}}(x_j)$ for large $p$ and  so  $B_{\frac{r}{2}}(y_{j,p})\subset B_{r}(x_j)\subset\Omega$ for $p$ large, namely, by scaling
\begin{equation}
\label{inclusioniSets2}
B_{ \frac{r}{ 2\varepsilon_{j,p} }}(0)\subset \Omega_{j,p},\quad \mbox{ for $p$ large }
\end{equation}
and
\[
p\int_{B_{r}(x_j)}u_p(x)^{p+1}\, dx\geq p\int_{B_{\frac{r}{2}}(y_{j,p})} u_p(x)^{p+1}\, dx
=  u_p(y_{j,p})^2 \int_{B_{\frac{r}{2\varepsilon_{j,p}}}(0)} \left(1+\frac{w_{j,p}(y)}{p}\right)^{p+1}\, dy.
\]
Passing to the limit as $p\rightarrow +\infty$, by (i), (iv) and Fatou's Lemma
\begin{eqnarray*}
\liminf_p p\int_{B_{r}(x_j)}u_p(x)^{p+1}\, dx &\geq&
\liminf_p p\int_{B_{\frac{r}{2}}(y_{j,p})}u_p(x)^{p+1}\, dx \geq\liminf_p u_p(y_{j,p})^2\int_{\R^2}e^{U(y)}\, dy\\&=&  8\pi\cdot \liminf_p u_p(y_{j,p})^2
\end{eqnarray*}
which gives (v), moreover by the previous relation
 \begin{eqnarray*}
 \liminf_p u_p(y_{j,p})\ p\int_{B_{r}(y_{j,p})}u_p(x)^{p}\, dx &\overset{\eqref{massimoVaAfinireInPallaPiccolaApiacere}}{\geq} &
\liminf_p p\int_{B_{r}(y_{j,p})}u_p(x)^{p+1}\, dx\geq
\liminf_p p\int_{B_{\frac{r}{2}}(y_{j,p})}u_p(x)^{p+1}\, dx\\
&& \geq 8\pi\cdot \liminf_p u_p(y_{j,p})^2.
\end{eqnarray*}
\end{proof}

\

\

\begin{proposition} \label{PropositionBetaConvergente}
Let $r>0$ be as in \eqref{rPiccoloAbbastanza} and define, for $j=1,\ldots,N$:
\begin{equation}\label{betap}
 \beta_{j,p}:=\frac{p}{u_p(y_{j,p})}\int_{B_{r}(y_{j,p})}u_p(x)^{p}\,dx.
 \end{equation} Then
\begin{equation}\label{limitedibeta}\lim_{p\rightarrow +\infty}\beta_{j,p}= 8\pi.
\end{equation}
\end{proposition}

\begin{proof}
Fix $j\in\{1,\ldots,N\}.$ By Lemma \ref{lemma:possoRiscalareAttornoAiMax}-$(vi)$ we already know  that
\[
\lim_{p\rightarrow +\infty}\beta_{j,p}\geq 8\pi,
\]
so we have to prove only the opposite inequality:
\begin{equation}\label{betaMin}
\lim_{p\rightarrow +\infty}\beta_{j,p}\leq 8\pi.
\end{equation}
 For $\delta\in (0,r)$ by \eqref{rPiccoloAbbastanza}
 \begin{equation}\label{pallettinaDentroIns}
 B_{\delta}(x_j)\subset\Omega
 \end{equation} and we define
 \begin{equation}\label{defalphaj}
 \alpha_{j,p}(\delta):=\frac{p}{u_p(y_{j,p})}\int_{B_{\delta}(x_j)}u_p(x)^{p}\,dx.
 \end{equation}
 In order to prove \eqref{betaMin} it is sufficient to show that
\begin{equation}\label{alphadelta}
\lim_{\delta\rightarrow 0}\lim_{p\rightarrow +\infty}\alpha_{j,p}(\delta)\leq 8\pi
\end{equation}
since \eqref{betaMin} will follow observing that
\begin{equation}\label{betaEAlfaHannoStessoLimite}
\beta_{j,p}= \alpha_{j,p}(\delta) + \frac{p}{u_p(y_{j,p})}\int_{B_{r}(y_{j,p})\setminus B_{\delta}(x_j)}u_p(x)^{p}\, dx =  \alpha_{j,p}(\delta) + o_p(1),
\end{equation}
where the second term goes to zero as $p\rightarrow +\infty$ because $y_{j,p}\in\overline{B_{2r}(x_j)}$. Indeed $B_{r}(y_{j,p})\setminus B_{\delta}(x_j)\subset B_{3r}(x_j)\setminus B_{\delta}(x_j)\subset\bar\Omega\setminus \mathcal S$ and  we know that for any compact subset of $\bar\Omega\setminus \mathcal S$ the limit \eqref{integraleVaAZeroSuCompatti} holds and $\liminf _pu_p(y_{j,p})\geq 1$ by \eqref{maxLocMag1}.
\\

In the rest of the proof we show \eqref{alphadelta}.\\
By Lemma \ref{lemma:sommeDiGreen} we have  that $pu_p\rightarrow \sum_{j=1}^N\gamma_j G(\cdot,x_j)$ in $C^2_{loc}(B_{r}(x_i)\setminus\{x_i\})$.
Moreover
it is easy to see that for $x\in B_{r}(x_i)\setminus\{x_i\}$
\begin{eqnarray}
&&\sum_{j=1}^N\gamma_j G(x,x_j)=\gamma_iG(x,x_i) + O(1)\nonumber
\\
&&\sum_{j=1}^N\gamma_j\nabla G(x,x_j)=\gamma_i \nabla G(x,x_i) + O(1).\label{sommG}
\end{eqnarray}
Furthermore $G(x,x_i)=\frac{1}{2\pi}\log\frac{1}{|x-x_i|}+ H(x,x_i)$ by \eqref{defH}, so that, by the regularity of $H$, if $\delta\in (0,r)$ is small enough and $x\in \overline{B_{\delta}(x_i)}\setminus\{x_i\}$, then
\begin{eqnarray}
&& G(x,x_i)= \frac{1}{2\pi}\log\frac{1}{|x-x_i|} + O(1)
\nonumber\\
&&\nabla G(x,x_i) = -\frac{1}{2\pi}\frac{x-x_i}{|x-x_i|^2}+ O(1).\label{derG}
\end{eqnarray}

Applying the local Pohozaev identity \eqref{Pohozaev} in the set $B_{\delta}(x_i)$ with $y=x_i$ we obtain (observe that if $\nu(x)$ is the outer unitary normal vector to $\partial B_{\delta}(x_i)$ in $x$ then $\langle x-x_i,\nu(x)\rangle=|x-x_i|=\delta$)
\begin{eqnarray}\label{quaQua}
\frac{2p^2}{p+1}\int_{ B_{\delta}(x_i)}u_p(x)^{p+1}dx
& = &
-\frac{\delta}{2} \int_{\partial B_{\delta}(x_i)}
| p\nabla u_p(x)  |^2   ds_x
+\delta  \int_{\partial B_{\delta}(x_i)}
\langle p\nabla u_p(x),\nu(x)  \rangle^2 ds_x
\nonumber\\
&&\ + \
\frac{p^2}{p+1} \delta\int_{\partial B_{\delta}(x_i)}u_p(x)^{p+1}  ds_x
\end{eqnarray}
Next we analyze the behavior of the three terms in the right hand side.

By the uniform convergence of the derivative of  $pu_p$  on compact sets combined with \eqref{sommG} and \eqref{derG}, passing to the limit we have

\[-\frac{\delta}{2} \int_{\partial B_{\delta}(x_i)}
| p\nabla u_p(x)  |^2   ds_x  \underset{p\rightarrow +\infty}{\longrightarrow}
-\frac{\delta}{2}\int_{\partial B_{\delta}(x_i)}\left(-\gamma_i\frac{1}{2\pi}\frac{ x-x_i }{|x-x_i|^2} + O(1)\right)^2 ds_x  =-\frac{\gamma_i^2}{4\pi} + O(\delta)
\]
\[\delta  \int_{\partial B_{\delta}(x_i)}
\langle p\nabla u_p(x),\nu(x)  \rangle^2 ds_x \underset{p\rightarrow +\infty}{\longrightarrow}
\delta\int_{\partial B_{\delta}(x_i)}\left(-\gamma_i\frac{1}{2\pi}\frac{\langle x-x_i,\nu(x)\rangle }{|x-x_i|^2} + O(1)\right)^2 ds_x  =\frac{\gamma_i^2}{2\pi} + O(\delta)\]
and also
\begin{eqnarray*}
&&\frac{p^2}{p+1}\delta  \int_{\partial B_{\delta}(x_i)}u_p(x)^{p+1}   ds_x
\leq\frac{2\pi p}{p+1} \delta^2  \|p u_p^{p+1}\|_{L^{\infty}(\partial B_{\delta}(x_i))}\overset{\eqref{pupVaAZeroUnifSuiCompatti}}{=} o_p(1)O({\delta}^2).
\end{eqnarray*}
So by \eqref{quaQua}
and recalling  the definition of $\alpha_{j,p}$
\begin{equation}
\label{firsto}
\alpha_{j,p}(\delta)u_p(y_{j,p})^2     \overset{\eqref{defalphaj}}{=} u_p(y_{j,p})\ p \int_{ B_{\delta}(x_i)}u_p(x)^{p}dx    \geq p\int_{ B_{\delta}(x_i)}u_p(x)^{p+1}dx\overset{\eqref{quaQua}}{=}\frac{\gamma_i^2}{8\pi} + O(\delta)+ o_p(1)
\end{equation}
but
\begin{equation}
\label{secondt}
\gamma_j\overset{\eqref{defgammaj} - \eqref{pallettinaDentroIns}}{=}\lim_{\delta\rightarrow 0}\lim_{p\rightarrow +\infty} p\int_{ B_{\delta}(x_i)} u(x)^p\,dx \overset{\eqref{defalphaj} }{=}\lim_{\delta\rightarrow 0}\lim_{p\rightarrow +\infty} \alpha_{j,p}(\delta)u_p(y_{j,p}).
\end{equation}
Combining \eqref{firsto} and \eqref{secondt} we get \eqref{alphadelta}.
\end{proof}

\

Now, for any $j=1,\ldots,N$ we derive a decay estimate for the rescaled function  $w_{j,p}(y)$ defined in \eqref{defRiscalataMax} for  $y\in\Omega_{j,p}$.\\

First recall that  $\forall R>0$ there exists  $p_R>1$ such that
\begin{equation}
\label{inclusioniSetsInParagrafo}
B_R(0)\subset B_{ \frac{r}{ \varepsilon_{j,p} }}(0)\subset B_{\frac{2r}{\varepsilon_{j,p}}}\left(\frac{x_j-y_{j,p}}{\varepsilon_{j,p}} \right)\subset \Omega_{j,p}\ \mbox{
 for }\ p\geq p_R
 \end{equation}
(indeed using \eqref{massimoVaAfinireInPallaPiccolaApiacere} we have that $y_{j,p}\in B_{r}(x_j)$ and so $B_{r}(y_{j,p})\subset B_{2r}(x_j)\subset\Omega$ for $p$ large). Then by definition
\begin{equation} \label{Wminoredi1}
0\leq \left( 1+\frac{w_{j,p}(z)}{p}\right)\leq 1,\ \mbox{ for any }z\in \subset B_{\frac{2r}{\varepsilon_{j,p}}}\left(\frac{x_j-y_{j,p}}{\varepsilon_{j,p}} \right).
\end{equation}

\

Moreover observe that
 $\beta_{j,p}$ defined in \eqref{betap} can be now rewritten as
\begin{equation}\label{betap2}
\beta_{j,p}=	\int_{B_{\frac{r}{\varepsilon_{j,p}}}(0)} \left( 1+\frac{w_{j,p}(z)}{p}\right)^{p}dz.
\end{equation}

\

Last we recall that by Proposition \ref{prop:BoundEnergia}-(iv) and  \eqref{maxLocMag1} we have that
\begin{equation}\label{energiaLimitataConW}
\int_{\Omega_{j,p}}\left(1+\frac{w_{j,p}(z)}{p}\right)^{p}dz=\frac{p}{u_p(y_{j,p})}\int_{\Omega} u_p(x)^pdx =O(1).
\end{equation}

\

\

\

\begin{lemma}
\label{lemma:Iacopetti}
For any $\varepsilon>0$, there exist $R_{\varepsilon}>1$ and $p_{\varepsilon}>1$ such that
\begin{equation}\label{stimaIacopetti}
w_{j,p}(y)\leq \left(\frac{\beta_{j,p}}{2\pi}-\varepsilon\right)\log\frac{1}{|y|}+C_{\varepsilon},\qquad\forall j=1,\ldots,N
\end{equation}
for some $C_{\varepsilon}>0$, provided $2R_\varepsilon\leq |y|\leq \frac{r}{\varepsilon_{j,p}}$ and $p\geq p_{\varepsilon}$.
\end{lemma}

\begin{proof}
Given $\varepsilon>0$, we can choose $R_\varepsilon>1$ such that
\[\int_{B_{R_{\varepsilon}}(0)}e^{U(z)}dz> 8\pi -\varepsilon.\]
The function $w_{j,p}$ is well defined in $B_{R_{\varepsilon}}(0)$ for $p$ large by \eqref{inclusioniSetsInParagrafo}, moreover by Fatou's lemma and \eqref{convRiscalateNeiMax}

\[\liminf_{p\to+\infty}\int_{B_{R_{\varepsilon}}(0)} \left(1+\frac{w_{j,p}(z)}{p}\right)^{p} dz\geq\int_{B_{R_{\varepsilon}}(0)}e^{U(z)}\,dz
\]
namely for $p_{\varepsilon}>1$ sufficiently large
\begin{equation}\label{5.16SW}
\int_{B_{R_{\varepsilon}}(0)}\left(1+\frac{w_{j,p}(z)}{p}\right)^{p}dz>8\pi-\varepsilon\qquad\mbox{for all $p\geq p_\varepsilon$.}
\end{equation}
Let $2R_{\varepsilon}\leq |y|\leq\frac{r}{\varepsilon_{j,p}}$. Observe that when $|z|\geq\frac{2r}{\varepsilon_{j,p}}$ then $2|y|\leq |z|$, hence
\[
\frac{2}{3}\leq\frac{|z|}{|z|+\frac{|z|}{2}}\leq\frac{|z|}{|z|+|y|}\leq          \frac{|z|}{|y-z|}\leq\frac{|z|}{|z|-|y|}\leq\frac{|z|}{|z|-\frac{|z|}{2}}=2,
\]
which implies
\begin{equation}
\label{stimalogo}
\log\frac{2}{3}\leq \log \frac{|z|}{|y-z|}\leq \log 2
\end{equation}
and so
\begin{eqnarray}\label{pezzoFuoriPallaNonConta}
\int_{\{ |z|\geq\frac{2r}{\varepsilon_{j,p}} \}\cap  \Omega_{j,p} } \log \frac{|z|}{|y-z|}\left(1+\frac{w_{j,p}(z)}{p}\right)^{p}dz \overset{\eqref{stimalogo}}{=}  O(1) \int_{\Omega_{j,p}}  \left(1+\frac{w_{j,p}(z)}{p}\right)^{p}dz
\overset{\eqref{energiaLimitataConW}}{=}  O(1).
\end{eqnarray}
By \eqref{problem} and the Green's function representation we have that for any $y\in {\Omega}_{j,p}$
\begin{eqnarray*}
u_p(\varepsilon_{j,p}y+y_{j,p})
&=&\int_{\Omega}G(\varepsilon_{j,p}y+y_{j,p},x)u_p(x)^pdx
\\
&=&
\varepsilon_{j,p}^2\int_{{\Omega}_{j,p}}G(\varepsilon_{j,p}y+y_{j,p},\varepsilon_{j,p}z+y_{j,p})u_p(\varepsilon_{j,p}z+y_{j,p})^pdz \\
&=&
\frac{u_p(y_{j,p})}{p}\int_{{\Omega}_{j,p}}G(\varepsilon_{j,p}y+y_{j,p},\varepsilon_{j,p}z+y_{j,p})
\left(1+\frac{w_{j,p}(z)}{p}\right)^{p}dz
\end{eqnarray*}
namely
\begin{equation}\label{greenRappW}
w_{j,p}(y) = -p + \int_{{\Omega}_{j,p}}G(\varepsilon_{j,p}y+y_{j,p},\varepsilon_{j,p}z+y_{j,p})
\left(1+\frac{w_{j,p}(z)}{p}\right)^{p}dz.
\end{equation}
As a consequence
\begin{eqnarray}\label{stimaDimWp}
w_{j,p}(y)&=&w_{j,p}(y)-w_{j,p}(0)
\nonumber
\\
&\overset{\eqref{greenRappW}}{=}&
\int_{  {\Omega}_{j,p} } \left[G(\varepsilon_{j,p}y+y_{j,p},\varepsilon_{j,p}z+y_{j,p})-G(y_{j,p},\varepsilon_{j,p}z+y_{j,p})\right] \left(1+\frac{w_{j,p}(z)}{p}\right)^{p}dz
\nonumber
\\
&\overset{\eqref{defH}}{=}&
\frac{1}{2\pi }\int_{{\Omega}_{j,p} }
\log\frac{|z|}{|y-z|}\left(1+\frac{w_{j,p}(z)}{p}\right)^{p}dz
  \ +
\nonumber
\\
&&
+ \
\int_{{\Omega}_{j,p} }  \left[H(\varepsilon_{j,p}y+y_{j,p},\varepsilon_{j,p}z+y_{j,p})-H(y_{j,p},\varepsilon_{j,p}z+y_{j,p})\right] \left(1+\frac{w_{j,p}(z)}{p}\right)^{p}dz
\nonumber
\\
&\overset{\eqref{energiaLimitataConW}}{=}&
\frac{1}{2\pi }\int_{{\Omega}_{j,p}}\log\frac{|z|}{|y-z|}\left(1+\frac{w_{j,p}(z)}{p}\right)^{p}dz \ \  + O(1)\nonumber
\\
&\overset{\eqref{pezzoFuoriPallaNonConta}}{=}&
\frac{1}{2\pi }\int_{\{|z|\leq\frac{2r}{\varepsilon_{j,p}}\}} \log\frac{|z|}{|y-z|}\left(1+\frac{w_{j,p}(z)}{p}\right)^{p}dz \ +\ O(1),
\end{eqnarray}
since $H$ is Lipschitz continuous and $|\varepsilon_{j,p}y|\leq r$.
Next, let us divide the last integral of \eqref{stimaDimWp} in the following way:

\begin{eqnarray}\label{secondacatena}
w_{j,p}(y) &\overset{\eqref{stimaDimWp}}{=}& \frac{1}{2\pi }\int_{\{|z|\leq R_{\varepsilon}\}} \log\frac{|z|}{|y-z|}\left(1+\frac{w_{j,p}(z)}{p}\right)^{p}dz
\nonumber
\\
&&
+\
\frac{1}{2\pi }\int_{\{R_{\varepsilon}\leq|z|\leq\frac{2r}{\varepsilon_{j,p}}\}\cap\{|z|\leq 2|y-z|\}} \log\frac{|z|}{|y-z|}\left(1+\frac{w_{j,p}(z)}{p}\right)^{p}dz
\nonumber
\\
&&
+\ \frac{1}{2\pi }\int_{\{R_{\varepsilon}\leq|z|\leq\frac{2r}{\varepsilon_{j,p}}\}\cap\{|z|\geq 2|y-z|\}} \log|z|\left(1+\frac{w_{j,p}(z)}{p}\right)^{p}dz
\nonumber
\\
&&
+\ \frac{1}{2\pi }\int_{\{R_{\varepsilon}\leq|z|\leq\frac{2r}{\varepsilon_{j,p}}\}\cap\{|z|\geq 2|y-z|\}} \log\frac1{|y-z|}\left(1+\frac{w_{j,p}(z)}{p}\right)^{p}dz \ +\ O(1)
\nonumber
\\
&=& I\ +\ II\ +\ III\ +\ IV\ +\ O(1).
\end{eqnarray}

\

In order to estimate the first integral in the right hand side of \eqref{secondacatena} we observe that if $|z|\leq R_\varepsilon$ then $2|z|\leq|y|$ and so
\[
\frac{|z|}{|y-z|}\leq \frac{|z|}{|y|-|z|}\leq \frac{|z|}{|y|-\frac{|y|}{2}}\leq\frac{2R_{\varepsilon}}{|y|},
\]
therefore
\begin{equation}\label{integrale1}
I=\frac{1}{2\pi }\int_{\{|z|\leq R_{\varepsilon}\}0} \log\frac{|z|}{|y-z|}\left(1+\frac{w_{j,p}(z)}{p}\right)^{p}dz\leq \frac{1}{2\pi }\log\frac{2R_\varepsilon}{|y|}\int_{\{|z|\leq R_{\varepsilon}\}} \left(1+\frac{w_{j,p}(z)}{p}\right)^{p}dz.
\end{equation}

Next, the second term in \eqref{secondacatena} can be trivially estimated as
\begin{eqnarray}\label{integrale2}
II&=& \frac{1}{2\pi }\int_{\{R_{\varepsilon}\leq|z|\leq\frac{2r}{\varepsilon_{j,p}}\}\cap\{|z|\leq 2|y-z|\}} \log\frac{|z|}{|y-z|}\left(1+\frac{w_{j,p}(z)}{p}\right)^{p}dz
\nonumber
\\
&&\leq
\frac{1}{2\pi }\log 2\int_{\{R_{\varepsilon}\leq|z|\leq\frac{2r}{\varepsilon_{j,p}}\}\cap\{|z|\leq 2|y-z|\}}\left(1+\frac{w_{j,p}(z)}{p}\right)^{p}dz
\nonumber\\
&&
\overset{\eqref{betap2}}{\leq}\frac{1}{2\pi }\log 2\left(\beta_{j,p}-\int_{B_{R_\varepsilon}(0)}\left(1+\frac{w_{j,p}(z)}{p}\right)^{p}dz   +\int_{\{\frac{r}{\varepsilon_{j,p}}\leq |z|\leq \frac{2r}{\varepsilon_{j,p}}\}}\left(1+\frac{w_{j,p}(z)}{p}\right)^{p}dz\right)
\nonumber\\
&&
 \overset{\eqref{maxLocMag1}}{\leq}
\frac{1}{2\pi }\log 2\left(\beta_{j,p}-\int_{B_{R_\varepsilon}(0)}\left(1+\frac{w_{j,p}(z)}{p}\right)^{p}dz   + p\int_{\{r\leq |x-y_{j,p}|\leq 2r\}}u_p(x)^{p}dx\right)
\nonumber\\
&&
\leq
\frac{1}{2\pi }\log 2\left(\beta_{j,p}-\int_{B_{R_\varepsilon}(0)}\left(1+\frac{w_{j,p}(z)}{p}\right)^{p}dz   + \varepsilon\right),
\end{eqnarray}
where in the last inequality we have used that since $y_{j,p}\rightarrow x_j$ then for $p$ large $\{r\leq |x-y_{j,p}|\leq 2r\}\subset K:= \{\frac{r}{2}\leq |x-x_j|\leq 3r\}\subset \overline\Omega\setminus\mathcal S $ and compact and so
\begin{equation}\label{ausss}
p\int_{r\leq |x-y_{j,p}|\leq 2r}u_p(x)^{p}dx\leq p\int_{K }u_p(x)^{p}dx\overset{\eqref{integraleVaAZeroSuCompatti} }{\leq} \varepsilon, \mbox{ for large $p$}.
\end{equation}

To deal with the third integral in the right hand side of \eqref{secondacatena} we notice that if $|z|\geq2|y-z|$, then
\[
|z|\geq2|y-z|\geq 2|z|-2|y|\qquad\mbox{and so}\qquad |z|\leq 2|y|,
\]
hence
\begin{eqnarray}\label{integrale3}
III&=&\frac{1}{2\pi }\int_{\{R_{\varepsilon}\leq|z|\leq\frac{2r}{\varepsilon_{j,p}}\}\cap\{|z|\geq 2|y-z|\}} \log|z|\left(1+\frac{w_{j,p}(z)}{p}\right)^{p}dz
\nonumber
\\
&&\leq
\frac{1}{2\pi }\log (2|y|)\int_{\{R_{\varepsilon}\leq|z|\leq\frac{2r}{\varepsilon_{j,p}}\}\cap\{|z|\geq 2|y-z|\}}\left(1+\frac{w_{j,p}(z)}{p}\right)^{p}dz
\nonumber
\\
&&
\overset{\eqref{betap2}}{\leq}\frac{1}{2\pi }\log (2|y|)\left(\beta_{j,p}-\int_{B_{R_\varepsilon}(0)}\left(1+\frac{w_{j,p}(z)}{p}\right)^{p}dz  + \int_{\{\frac{r}{\varepsilon_{j,p}}\leq |z|\leq \frac{2r}{\varepsilon_{j,p}}\}}\left(1+\frac{w_{j,p}(z)}{p}\right)^{p}dz\right)
\nonumber\\
&&
\overset{\eqref{ausss}}{\leq}
\frac{1}{2\pi }\log (2|y|)\left(\beta_{j,p}-\int_{B_{R_\varepsilon}(0)}\left(1+\frac{w_{j,p}(z)}{p}\right)^{p}dz   + \varepsilon\right).
\end{eqnarray}
Finally we estimate the fourth integral. Observe that
for any  $y\in\{2R_{\varepsilon}\leq |y|\leq \frac{r}{\varepsilon_{j,p}}\}$ one has the inclusion
\[
\left\{R_{\varepsilon}\leq|z|\leq\frac{2r}{\varepsilon_{j,p}},\ |y-z|\leq 1\right\}\subset B_{\frac{r}{\varepsilon_{j,p}}+1}(0)
\]
and that, since for $p$ large enough $B_{r+\varepsilon_{j,p}}(y_j)\subset B_{2r}(x_j)$, then by scaling, also
\[B_{\frac{r}{\varepsilon_{j,p}}+1}(0)\subset B_{\frac{2r}{\varepsilon_{j,p}}}\left(\frac{x_j-y_{j,p}}{\varepsilon_{j,p}} \right),\]
so that, as a consequence, the estimate in \eqref{Wminoredi1}  holds for
$z\in \left\{R_{\varepsilon}\leq|z|\leq\frac{2r}{\varepsilon_{j,p}},\ |y-z|\leq 1\right\}$.
Hence
\begin{eqnarray}\label{integrale4}
IV&=&\frac{1}{2\pi }\int_{\{R_{\varepsilon}\leq|z|\leq\frac{2r}{\varepsilon_{j,p}}\}\cap\{|z|\geq 2|y-z|\}} \log\frac1{|y-z|}\left(1+\frac{w_{j,p}(z)}{p}\right)^{p}dz
\nonumber\\
&=&\frac{1}{2\pi }\int_{\{R_{\varepsilon}\leq|z|\leq\frac{2r}{\varepsilon_{j,p}}\}\cap\{|y-z|\leq 1\}} \log\frac1{|y-z|}\left(1+\frac{w_{j,p}(z)}{p}\right)^{p}dz \nonumber\\
&&+\frac{1}{2\pi }\int_{\{R_{\varepsilon}\leq|z|\leq\frac{2r}{\varepsilon_{j,p}}\}\cap\{2\leq2|y-z|\leq |z|\}} \log\frac1{|y-z|}\left(1+\frac{w_{j,p}(z)}{p}\right)^{p}dz\nonumber\\
&\overset{\eqref{Wminoredi1}     }{\leq}& \frac{1}{2\pi }\int_{\{|y-z|\leq 1\}} \log\frac1{|y-z|}dz=O(1),
\end{eqnarray}
where we have also used that $\log t\leq 0 $ if $t\leq1$.\\
At last, substituting \eqref{integrale1}, \eqref{integrale2}, \eqref{integrale3}, \eqref{integrale4} into \eqref{secondacatena},  using \eqref{5.16SW}, that $\frac{2R_\varepsilon}{|y|}\leq1$ and observing that, by Proposition \ref{PropositionBetaConvergente},  $|\beta_{j,p}-8\pi|<\varepsilon$ for $p$ large, we obtain the thesis, indeed:
\begin{eqnarray}
w_{j,p}(y) &\leq& \frac{1}{2\pi }\log\frac{2R_\varepsilon}{|y|}\ (8\pi-\varepsilon)
\nonumber
\\
&&+\frac{1}{2\pi }\log2\ \left(\beta_{j,p}-(8\pi-\varepsilon) {\color{blue} + \varepsilon}\right)\nonumber
\\
&&+\frac{1}{2\pi }\log(2|y|)\ \left(\beta_{j,p}-(8\pi-\varepsilon){\color{blue} + \varepsilon}\right)+O(1)\nonumber
\\
&\leq&\frac1{2\pi}(\beta_{j,p}-2(\beta_{j,p}-8\pi)- {\color{blue}3}\varepsilon)\ \log\frac1{|y|}+O(1)
\nonumber
\\
&\leq&\left(\frac{\beta_{j,p}}{2\pi}-\varepsilon\right)\log\frac1{|y|}+O(1).
\end{eqnarray}
\end{proof}

\

\

\section{Conclusion of the Proof of Theorem \ref{teo:Positive}}\label{section:conclusion}

\

Let $r>0$ be as in \eqref{rPiccoloAbbastanza} and let $y_{j,p}$ for $j=1,\ldots,N$ be the local maxima of $u_p$ as in \eqref{defy}.
Let us define
\begin{equation}\label{mSec}
m_j:=\lim\limits_{p\to+\infty}u_p(y_{p,j})=\lim\limits_{p\to+\infty}\|u_p(x)\|_{L^\infty(\overline{B_{2r}(x_j)})},\quad\mbox{ for }j=1,\ldots, N
\end{equation}
$\;$\\
(observe that  \eqref{boundSoluzio} implies that $m_j<+\infty$ for any $j=1,\ldots,N$).

\

\begin{proposition}
\label{proposizione:quasiQuanteConclusione}
One has
\begin{equation}\label{valoreGamma}
\gamma_j=8\pi\cdot m_j
\end{equation}
($\gamma_j$ defined by Lemma \ref{lemma:sommeDiGreen})
\begin{equation}
\label{energiaSemiQuantizzata}
\lim_{p\rightarrow +\infty} p\int_{\Omega}|\nabla u_p|^2= 8\pi \sum_{j=1}^N m_j^2
\end{equation}

\begin{equation}
\label{N=k}
N=k.
\end{equation}
where $k$ is the integer in Proposition \ref{thm:x1N}.
\end{proposition}

\begin{proof}
We prove \eqref{valoreGamma}. From the expression of $\gamma_j$ given by Lemma \ref{lemma:sommeDiGreen} combined with Proposition \ref{proposition:noBoundary} and the results in  \eqref{betaEAlfaHannoStessoLimite}, we have
\[\gamma_j=\lim_{\delta\rightarrow 0}\lim_{p\rightarrow +\infty}p\int_{B_{\delta}(x_j)} u(x)^p\,dx =\lim_{\delta\rightarrow 0}\lim_{p\rightarrow +\infty}  \alpha_{j,p}(\delta)u_p(y_{j,p})\overset{\eqref{betaEAlfaHannoStessoLimite}}{=} \lim_{p\rightarrow +\infty}\beta_{j,p}u_p(y_{j,p})=8\pi\cdot m_j,\]
where the last equality follows from Proposition \ref{PropositionBetaConvergente}.

\

Next we prove \eqref{energiaSemiQuantizzata}. Observe that
\begin{eqnarray}\label{centerC}
p\int_{\Omega}|\nabla u_p|^2&=&
p\int_{\Omega}u_p(x)^{p+1} dx= \sum_{j=1}^N p\int_{B_r(x_j)}u_p(x)^{p+1}\ dx +p\int_{\Omega\setminus \cup_{j=1}^N B_r(x_j)} u_p(x)^{p+1}\ dx
\nonumber\\
&\overset{\eqref{integraleVaAZeroSuCompatti}}{=} &
 \sum_{j=1}^N p\int_{B_r(x_j)}u_p(x)^{p+1}\ dx + o_p(1).
\end{eqnarray}
Moreover
\begin{equation}\label{perdidi}
p\int_{B_r(x_j)}u_p(x)^{p+1}\ dx= p\int_{B_{\frac{r}{2}}(y_{j,p})}u_p(x)^{p+1}\ dx + o_p(1)%=
% (8\pi+ o_p(1))u_p(y_{j,p})^2 + o_p(1)
,
 \end{equation}
since  for $p$ large enough $B_{\frac{r}{3}}(x_j)\subset B_{\frac{r}{2}}(y_{j,p})\subset B_r(x_j)$ so that
\[
p\int_{B_{r}(x_j)\setminus {B_{ \frac{r}{2}}(y_{j,p})} }u_p(x)^{p+1}dx
\leq p\int_{\{\frac{r}{3}<|x-x_j|<r\}}u_p(x)^{p+1}dx \overset{\eqref{integraleVaAZeroSuCompatti} }{=} o_p(1).
\]
Let us consider the  remaining term in the right hand side of \eqref{perdidi}. We want to rescale $u_p$ around the maximum point $y_{j,p}$ defining $w_{j,p}$ as in  \eqref{defRiscalataMax} and then pass to the limit.

Observe that, since by definition $w_{j,p}\leq 0$ , then
\begin{equation}
\label{arietta2} 0\leq  \left( 1+\frac{w_{j,p}(z)}{p}\right)^{p+1}=
 e^{(p+1)\log \left( 1+\frac{w_{j,p}(z)}{p} \right)}
 \leq e^{(p+1)\frac{w_{j,p}(z)}{p} }
 =   e^{w_{j,p}(z)}\leq C\frac{1}{|z|^3}.
\end{equation}
where the last inequality is due to the combination of  Lemma \ref{lemma:Iacopetti} and Proposition \ref{PropositionBetaConvergente} so that it holds
 provided $2R_\varepsilon\leq |z|\leq \frac{r}{\varepsilon_{j,p}}$ and $p$ is sufficiently large.
 Instead, when $|z|\leq 2R_{\varepsilon}$,  then by \eqref{inclusioniSetsInParagrafo} and \eqref{Wminoredi1} for $p$ sufficiently large
\begin{equation}\label{dentroReps2}
0\leq \left(1+\frac{w_{j,p}(z)}{p}\right)^{p+1}\leq 1,\quad\mbox{ for }|z|\leq 2R_{\varepsilon}.
\end{equation}
Thanks to \eqref{arietta2} and  \eqref{dentroReps2} we can apply the dominated convergence theorem deducing that
\begin{equation}
\label{primoPezzoASxprece}
\frac{p}{u_p(y_{j,p})^2}\int_{B_{\frac{r}{2}}(y_{j,p})}u_p(x)^{p+1}dx
=  \int_{B_{\frac{r}{2\varepsilon_{j,p}}}(0)} \left(1+\frac{w_{j,p}(z)}{p}\right)^{p+1}dz \underset{p\rightarrow +\infty}{\longrightarrow} \int_{\R^2}e^{U(z)}dz=8\pi.
\end{equation}
Substituting \eqref{primoPezzoASxprece}  into \eqref{perdidi} we then have
\begin{equation}\label{perPartiElaboratastar}
p\int_{B_r(x_j)}u_p(x)^{p+1}dx= (8\pi +o_p(1))\ m_j^2 + o_p(1).
\end{equation}

By substituting \eqref{perPartiElaboratastar} into \eqref{centerC} we get \eqref{energiaSemiQuantizzata}.

\

Finally we show \eqref{N=k}.
\\
We have defined $N\in \N\setminus\{0\}$ to be the number of points in the concentration set $\mathcal{S}$, hence $N\leq k$.\\
Recall that in \eqref{ribattezzo} w.l.g. we have relabeled the sequences of points $x_{i,p}$, $i=1,\ldots, k$ in such a way that
\[
x_{j,p}\rightarrow x_j,\ \forall \,j=1,\ldots, N\quad\mbox{ and }\quad\mathcal{S}=\{ x_1, \ x_2, \ \ldots, \ x_N \}
\]
and without loss of generality we may also assume that
\[R_{k,p}(y_{j,p})=|x_{j,p}-y_{j,p}|.\]
Then Proposition \ref{lemma:rapportoMuBounded} applied to the family $y_{j,p}$ implies that $
\limsup_{p\rightarrow +\infty}\frac{\mu_{j,p}}{\varepsilon_{j,p}}\leq  1,
$
but, by the definition of $y_{j,p}$ as a maximum of $u_p$ on $B_{2r}(x_j)$ and the fact that $x_{j,p}\rightarrow x_j$, we also have that $\frac{\mu_{j,p}}{\varepsilon_{j,p}}\geq  1$ for $p$ large. As a consequence
\begin{equation}
\label{coincidenza}
\lim_{p\rightarrow +\infty} u_p(x_{j,p})=\lim_{p\rightarrow +\infty} u_p(y_{j,p})=m_j, \quad\forall j=1,\ldots, N.
\end{equation}
Assume now by contradiction that $N<K$. Since, by Proposition \ref{thm:x1N}, $(\mathcal{P}_1^k)$ and $(\mathcal{P}_2^k)$ hold then Lemma \ref{lemma:BoundEnergiaBassino} applies and so we have
\begin{eqnarray*}
\lim_{p\rightarrow +\infty} p\int_\Omega |\na\upp|^2\,dx &\geq & 8\pi \sum_{i=1}^k \lim_p \left(u_p(x_{i,p})\right)^2\overset{\eqref{coincidenza}}{=} 8\pi \sum_{j=1}^N m_j^2 + 8\pi\!\!\!\!\sum_{i=N+1}^k \lim_p \left(u_p(x_{i,p})\right)^2 \\
&\overset{\eqref{RemarkMaxCirca1}}{\geq}& 8\pi \sum_{j=1}^N m_j^2 + 8\pi > 8\pi \sum_{j=1}^N m_j^2.
\end{eqnarray*}
 which contradicts \eqref{energiaSemiQuantizzata} and this concludes the proof.
\end{proof}

\

\

By the results in Section \ref{section:noconcentrationboundary}  and \eqref{N=k} in Proposition \ref{proposizione:quasiQuanteConclusione} we have that $\mathcal S=\{x_1, x_2,\cdots, x_k\}\subset\Omega$. We now locate the concentration points $x_i, \ i=1,\ldots, k$.

\

\begin{proposition}
\label{prop:equazionePuntiConcentrazione}
The concentration points $x_i, \ i=1,\ldots, k$, satisfy \eqref{x_j relazione}, namely
\[
m_i \nabla_x H(x_i,x_i)+\sum_{\ell\neq i} m_\ell \nabla_x G(x_i,x_\ell)=0.
\]
\end{proposition}
\begin{proof} Let $\delta>0$ small enough so that $B_{\delta}(x_i)\subset\Omega$ and $B_{\delta}(x_i)\cap B_{\delta}(x_j)=\emptyset$, $i\neq j$. Clearly it is enough to prove the identity for $i=1$.\\
Multiplying equation \eqref{problem} by $\frac{\partial u_p}{\partial x_j}$, for $j=1,2$, and integrating on $B_{\delta}(x_1)$ we have that
\begin{equation}\label{intbyparts}
-\int_{B_{\delta}(x_1)} \Delta u\frac{\partial u}{\partial x_j}dx=\int_{B_{\delta}(x_1)}|u|^{p-1}u\frac{\partial u}{\partial x_j}dx=\frac{1}{p+1}\int_{B_{\delta}(x_1)} \frac{\partial}{\partial x_j}|u|^{p+1}dx=\frac{1}{p+1}\int_{\partial B_{\delta}(x_1)} |u|^{p+1}\nu_j ds_x
\end{equation}
where $\nu_j$ are the components of the outer normal at $\partial B_{\delta}(x_1)$.\\
For the first term
\begin{align}\label{primotermine}
\int_{B_{\delta}(x_1)}\Delta u\frac{\partial u}{\partial x_j}dx &=\int_{B_{\delta}(x_1)}\sum_{h}\frac{\partial^2 u}{\partial x_h^2}\frac{\partial u}{\partial x_j}dx
\nonumber
\\
&=\int_{B_{\delta}(x_1)}\sum_h\frac{\partial}{\partial x_h}\left(\frac{\partial u}{\partial x_j}\frac{\partial u}{\partial x_h}\right)dx-\int_{B_{\delta}(x_1)}\sum_h\left(\frac{\partial^2 u}{\partial x_h \partial x_j}\frac{\partial u}{\partial x_h}\right)dx
\nonumber
\\
&=\int_{\partial B_{\delta}(x_1)}\sum_h\left(\frac{\partial u}{\partial x_j}\frac{\partial u}{\partial x_h}\right)\,\nu_hds_x -\frac12\int_{B_{\delta}(x_1)} \frac{\partial}{\partial x_j}\left(\sum_h\left(\frac{\partial u}{\partial x_h}\right)^2\right)dx
\nonumber
\\
&=\int_{\partial B_{\delta}(x_1)}\frac{\partial u}{\partial x_j}\frac{\partial u}{\partial \nu}ds_x-\frac12\int_{\partial B_{\delta}(x_1)}|\nabla u|^2\nu_j ds_x.
\end{align}
Hence, combining \eqref{intbyparts} with \eqref{primotermine} and multiplying by $p^2$ we get:
\begin{equation}
\label{prima}
\frac{p^2}{p+1}\int_{\partial B_{\delta}(x_1)}|u_p|^{p+1}\nu_jds_x +p^2\int_{\partial B_{\delta}(x_1)}\frac{\partial u_p}{\partial x_j}\frac{\partial u_p}{\partial \nu}ds_x-\frac{p^2}2\int_{\partial B_{\delta}(x_1)}|\nabla u_p|^2\nu_jds_x=0.
\end{equation}

The first term of \eqref{prima} can be estimated as follows
\[
\left|\frac{p^2}{p+1}\int_{\partial B_{\delta}(x_1)}|u_p|^{p+1}\nu_j ds_x \right|\leq\frac{p}{p+1}2\pi\delta\|pu_p^{p+1}\|_{L^\infty(\partial B_{\delta}(x_1))}\overset{\eqref{pupVaAZeroUnifSuiCompatti}}{\underset{p\to+\infty}{\longrightarrow}}0,
\]
so letting $p\to+\infty$ in \eqref{prima} we get, for $j=1,2$,
\[
\int_{\partial B_{\delta}(x_1)}\frac{\partial}{\partial x_j}\left(\sum_{i=1}^N m_i G(x,x_i)\right)\frac{\partial}{\partial\nu}\left(\sum_{i=1}^N m_i G(x,x_i)\right)ds_x-\frac{1}{2}\int_{\partial B_{\delta}(x_1)}\left|\nabla\left(\sum_{i=1}^N m_i G(x,x_i)\right)\right|^2\nu_jds_x=0,
\]
where we have used Theorem \ref{teo:Positive}-\emph{(ii)} (which holds by virtue of Lemma \ref{lemma:sommeDiGreen} and Proposition \ref{proposizione:quasiQuanteConclusione}).\\
Computing the last integral as in \cite[pp. 511-512]{MaWei} we obtain
\[
-m_1\left[m_1\frac{\partial}{\partial x_j}H(x_{j,r}^*,x_1)+\sum_{\ell\neq 1}m_\ell\frac{\partial}{\partial x_j}G(x_{j,r}^*,x_\ell)\right]+o_{\delta}(1)=0\qquad\mbox{for $j=1,2$,}
\]
where $x_{j,r}^*\to x_1$ as ${\delta}\to0$, $j=1,2$. Hence passing to the limit as ${\delta}\to0$ we derive the desired relations
\[
m_1\frac{\partial}{\partial x_j}H(x_1,x_1)+\sum_{\ell\neq 1}m_\ell\frac{\partial}{\partial x_j}G(x_1,x_\ell)=0\qquad\mbox{for $j=1,2$.}
\]
\end{proof}

\

We now estimate from below the numbers $m_j$, $j=1,\ldots, N$ in \eqref{mSec} and so the
 $L^{\infty}$-norm of $u_p$. We will need the following result
\begin{lemma}{(\cite[Lemma 2.1]{RenWeiTAMS1994})}
Let $B\subset\R^2$ be a smooth bounded domain. Then for every $p>1$ there exists $D_p>0$ such that
\begin{equation}\label{stimaRenWey}
\|v\|_{L^{p+1}(B)}\leq D_p (p+1)^{1/2}\|\nabla v\|_{L^2(B)}, \quad\forall v\in H^1_0(B).
\end{equation}
Moreover
\begin{equation}
\label{Dptende}
\lim_{p\rightarrow +\infty} D_p=\frac{1}{(8\pi e)^{1/2}}.
\end{equation}
\end{lemma}

\

\begin{proposition}
\label{prop:maxGeqE}
We have
\[m_j \geq \sqrt{e},\qquad \forall j=1,\ldots,N.\]
and hence
\[
\lim\limits_{p\rightarrow +\infty} \|u_p\|_{\infty}\geq \sqrt{e}.
\]
\end{proposition}
\begin{proof}
Let $r$ be as in \eqref{rPiccoloAbbastanza}. Let us take $\chi\in C^2([0,2r))$, $\chi(s)=1$ for $s\in [0,r)$, $\chi(s)=0$ for $s\in [\frac{3}{4}r,2r)$ and consider the cut-off function $\chi_j(x):=\chi(|x-x_j|)$.
Then $\widetilde u_{p,j}:=u_p\chi_j\in H^1_0(B_{2r}(x_j))$ and satisfies
\begin{eqnarray}
\label{energiautildeapproxu}
p\int_{B_{2r}(x_j)}|\nabla \widetilde u_{p,j}|^2\,dx &=& p\int_{B_{r}(x_j)}|\nabla u_{p}|^2\,dx +
p\int_{\{r\leq |x-x_j|\leq \frac{3}{4}r\}}|\nabla u_{p}|^2\chi_j^2\,dx
\nonumber
\\
&& + p\int_{\{r\leq |x-x_j|\leq \frac{3}{4}r\}} u_p^2|\nabla\chi_j|^2\,dx \nonumber
\\
&&+2p\int_{\{r\leq |x-x_j|\leq \frac{3}{4}r\}} u_p\langle\nabla u_p,\nabla\chi_j\rangle\chi_j\, dx
\nonumber
\\
&\overset{\eqref{pu_va_a_zero}}{=}& p\int_{B_{r}(x_j)}|\nabla u_{p}|^2\,dx + o_p(1).
\end{eqnarray}
Moreover
\begin{eqnarray}
\label{energiautildeapproxuallap}
p\int_{B_{2r}(x_j)}\widetilde u_{p,j}^{p+1}\,dx &=&
p\int_{B_{r}(x_j)}u_{p}^{p+1}\,dx + p\int_{\{r\leq |x-x_j|\leq \frac{3}{4}r\}} u_{p}^{p+1}\chi_j^{p+1}\,dx
\nonumber
\\
&\overset{\eqref{pupVaAZeroUnifSuiCompatti}}{=} &
p\int_{B_{r}(x_j)}u_{p}^{p+1}\,dx + o_p(1).
\end{eqnarray}
By \eqref{energiautildeapproxu} and \eqref{energiautildeapproxuallap}, applying \eqref{stimaRenWey} to $\widetilde u_{p,j}$ and using \eqref{Dptende}, we then get
\begin{eqnarray}
\label{topolino}
p\int_{B_r(x_j)}|\nabla u_p|^2\,dx &\overset{\eqref{energiautildeapproxu}}{=}&
p\int_{B_{2r}(x_j)}|\nabla \widetilde u_{p,j}|^2\,dx +o_p(1)
\nonumber\\
 &\overset{\eqref{stimaRenWey}}{\geq}&
 \frac{p}{(p+1)p^{\frac{2}{p+1}}D_p^2} \left[p\int_{B_{2r}(x_j)}\widetilde u_{p,j}^{p+1}\,dx\right]^{\frac{2}{p+1}} +o_p(1)
\nonumber\\
&\overset{\eqref{Dptende}}{=}& (8\pi e +o_p(1))   \left[p\int_{B_{2r}(x_j)}\widetilde u_{p,j}^{p+1}\,dx\right]^{\frac{2}{p+1}} +o_p(1)
\nonumber\\
&\overset{\eqref{energiautildeapproxuallap}}{=}& (8\pi e +o_p(1))
\left[p\int_{B_{r}(x_j)} u_{p}^{p+1}\,dx+o_p(1)\right]^{\frac{2}{p+1}} +o_p(1)
\nonumber\\
&{=}& 8\pi e +o_p(1)
 \end{eqnarray}
where in the last line we have used that $c\leq p\int_{B_{r}(x_j)} u_{p}^{p+1}\,dx\leq C$ which follows by the assumption on the energy bound \eqref{energylimit} and from \eqref{energiaBPallaLontanaDaZero}.
Finally observe that integrating by parts and using the equation \eqref{problem} we also have
\begin{eqnarray}
\label{minni}
p\int_{B_r(x_j)}|\nabla u_p|^2\,dx &=& p\int_{B_r(x_j)} u_p^{p+1}\,dx -p\int_{\partial B_r(x_j)} u_p\frac{\partial u_p}{\partial\nu}
\nonumber\\
&\overset{\eqref{pu_va_a_zero}}{=}& p\int_{B_r(x_j)} u_p^{p+1}\,dx + o_p(1)
\nonumber\\
&\overset{\mbox{\footnotesize{ as in }}\eqref{perdidi}}{=}&
 p\int_{B_{\frac{r}{2}}(y_{j,p})}u_p(x)^{p+1}\ dx + o_p(1)
\nonumber\\
&\overset{\eqref{defRiscalataMax}}{=}& u_p(y_{j,p})^2\int_{B_{\frac{r}{2\varepsilon_{j,p}}}(0)} \left(1+\frac{w_{j,p}(z)}{p}  \right)^{p+1}\,dz
\nonumber\\
&\overset{(*)}{=}& u_p(y_{j,p})^2 \left(o_p(1) + \int_{\R^2}e^{U(z)}dz \right)
\nonumber\\
&=& u_p(y_{j,p})^2 \ 8\pi +o_p(1),
\end{eqnarray}
where $(*)$ follows by the dominated convergence theorem in the same way as in \eqref{primoPezzoASxprece}.\\
Putting together \eqref{topolino} and \eqref{minni} we get the conclusion.
\end{proof}

\

\begin{proof}[Proof of Theorem \ref{teo:Positive}]
The statements of Theorem \ref{teo:Positive} have been proved in the various propositions obtained so far. In particular \emph{(i)} is the statement \eqref{pu_va_a_zero} in Proposition \ref{thm:x1N}, \emph{(ii)} derives from Lemma \ref{lemma:sommeDiGreen} and \eqref{valoreGamma} and \eqref{N=k} of Proposition \ref{proposizione:quasiQuanteConclusione}. The energy limit \emph{(iii)} is claim \eqref{energiaSemiQuantizzata} in Proposition \ref{proposizione:quasiQuanteConclusione}, together with \eqref{N=k}. The statement \emph{(iv)} is the assertion of Proposition \ref{prop:equazionePuntiConcentrazione} and \emph{(v)} is proved in Proposition \ref{prop:maxGeqE}.
\end{proof}

\

\

\end{document}